\documentclass[a4paper,reqno]{amsart}

\usepackage{setspace}
\usepackage[totalwidth=14cm,totalheight=22cm]{geometry}
\usepackage[T1]{fontenc}
\usepackage[dvips]{graphicx}
\usepackage{epsfig}
\usepackage{amsmath,amssymb,amscd,amsthm}
\usepackage{subfigure}
\usepackage[latin1]{inputenc}
\usepackage{latexsym}
\usepackage{graphics}
\usepackage{colortbl}
\usepackage{color}
\usepackage{ae}
\usepackage{enumitem}
\usepackage[all]{xy}
\usepackage{scalerel}
\usepackage{tikz}
\usepackage{cancel}

\makeatletter
\newtheorem*{rep@theorem}{\rep@title}
\newcommand{\newreptheorem}[2]{%
\newenvironment{rep#1}[1]{%
 \def\rep@title{#2 \ref{##1}}%
 \begin{rep@theorem}}%
 {\end{rep@theorem}}}
\makeatother

\makeatletter
\newtheorem*{rep@cor}{\rep@title}
\newcommand{\newrepcor}[2]{%
\newenvironment{rep#1}[1]{%
 \def\rep@title{#2 \ref{##1}}%
 \begin{rep@cor}}%
 {\end{rep@cor}}}
\makeatother

\makeatletter
\newtheorem*{rep@prop}{\rep@title}
\newcommand{\newrepprop}[2]{%
\newenvironment{rep#1}[1]{%
 \def\rep@title{#2 \ref{##1}}%
 \begin{rep@prop}}%
 {\end{rep@prop}}}
\makeatother

\newtheorem{cor}{Corollary}[section]
\newtheorem{corx}{Corollary}

\newtheorem{theorem}[cor]{Theorem}
\newtheorem{thmx}[corx]{Theorem}

\newtheorem{prop}[cor]{Proposition}
\newtheorem{propx}[corx]{Proposition}

\newrepcor{cor}{Corollary}
\newreptheorem{theorem}{Theorem}
\newrepprop{prop}{Proposition}

\newtheorem{lemma}[cor]{Lemma}

\theoremstyle{definition}
\newtheorem{defi}[cor]{Definition}
\theoremstyle{remark}
\newtheorem{remark}[cor]{Remark}
\newtheorem*{remark*}{Remark}
\newtheorem{example}[cor]{Example}

\newcommand{\dwp}{d_{\mathrm{W}\!\mathrm{P}}}
\newcommand{\weil}{{\mathrm{W}\!\mathrm{P}}}
\newcommand{\Th}{{\mathrm{T}\mathrm{h}}}

\newcommand{\N}{{\mathbb N}}
\newcommand{\R}{{\mathbb R}}
\newcommand{\Z}{{\mathbb Z}}

\newcommand{\Hyp}{\mathbb{H}}
\newcommand{\AdS}{\mathbb{A}\mathrm{d}\mathbb{S}}

\newcommand{\SL}{\mathrm{SL}}
\newcommand{\PSL}{\mathrm{PSL}}

\newcommand{\injrad}{\mathrm{injrad}}
\newcommand{\arcsinh}{\mathrm{arcsinh}}
\newcommand{\thin}{\mathrm{thin}}
\newcommand{\thick}{\mathrm{thick}}

\newcommand{\Diffeo}{\mbox{Diff}}

\newcommand{\trace}{\mbox{\rm tr}}
\newcommand{\tr}{\mbox{\rm tr}}

\newcommand{\grad}{\operatorname{grad}}
\newcommand{\isom}{\mathrm{Isom}}

\newcommand{\Vol}{\mathrm{Vol}}
\newcommand{\diam}{\mathrm{diam}}
\newcommand{\Area}{\mathrm{Area}}
\newcommand{\dth}{d_{\mathrm{Th}}}

\newcommand{\Ker}{\mathrm{Ker}}
\newcommand{\Ima}{\mathrm{Im}}

\def\Teich{\mathcal{T}}

\begin{document}

\setcounter{secnumdepth}{3}
\setcounter{tocdepth}{2}

\title[Volume of Anti-de Sitter $3$-manifolds]{On the volume of Anti-de Sitter maximal globally hyperbolic three-manifolds}

\author[Francesco Bonsante]{Francesco Bonsante}
\address{Francesco Bonsante: Dipartimento di Matematica ``Felice Casorati", Universit\`{a} degli Studi di Pavia, Via Ferrata 5, 27100, Pavia, Italy.} \email{bonfra07@unipv.it} 
\author[Andrea Seppi]{Andrea Seppi}
\address{Andrea Seppi: Dipartimento di Matematica ``Felice Casorati", Universit\`{a} degli Studi di Pavia, Via Ferrata 5, 27100, Pavia, Italy.} \email{andrea.seppi01@ateneopv.it}
\author[Andrea Tamburelli]{Andrea Tamburelli}
\address{Andrea Tamburelli: Mathematics Research Unit, University of Luxembourg, Campus Kirchberg, 6 rue Coudenhove-Kalergi, L-1359 Luxembourg.} \email{andrea.tamburelli@uni.lu}


\thanks{The authors were partially supported by FIRB 2010 project ``Low dimensional geometry and topology'' (RBFR10GHHH003). The first author was partially supported by
PRIN 2012 project ``Moduli strutture algebriche e loro applicazioni''.
The first two authors are members of the national research group GNSAGA}

\begin{abstract}
We study the volume of maximal globally hyperbolic Anti-de Sitter manifolds containing a closed orientable Cauchy surface $S$, in relation to some geometric invariants depending only on the two points in Teichm\"uller space of $S$ provided by Mess' parameterization - namely on two isotopy classes of hyperbolic metrics $h$ and $h'$ on $S$. The main result of the paper is that the volume coarsely behaves like the minima of the $L^1$-energy of maps from $(S,h)$ to $(S,h')$. 

The study of $L^p$-type energies had been suggested by Thurston, in contrast with the well-studied Lipschitz distance. A corollary of our result shows that the volume of maximal globally hyperbolic Anti-de Sitter manifolds is bounded from above by the exponential of (any of the two) Thurston's Lipschitz asymmetric distances, up to some explicit constants. Although there is no such bound from below, we provide examples in which this behavior is actually realized. We prove instead that the volume is bounded from below by the exponential of the Weil-Petersson distance.

The proof of the main result uses more precise estimates on the behavior of the volume, which is proved to be coarsely equivalent to the length of the (left or right) measured geodesic lamination of earthquake from $(S,h)$ to $(S,h')$, and to the minima of the holomorphic 1-energy. 
\end{abstract}

\maketitle

\section{Introduction}

Since the work of Mess in 1990 (\cite{Mess}), Anti-de Sitter geometry in dimension 3 has been extensively studied, two main motivations being its strong relations with the Teichm\"uller theory of hyperbolic surfaces, and its similarities with hyperbolic geometry. These are also the main motivations behind this paper. See \cite{Schlenker-Krasnov,bon_schl,bonschlfixed,BonSchlGAFA2009,bbzads,BonSepKsurfacesAdS,notes}.

In fact, the present work is concerned with the study of maximal globally hyperbolic spatially compact Anti-de Sitter 3-manifolds. These manifolds are Lorentzian manifolds of constant sectional curvature $-1$, are diffeomorphic to $S\times\R$, where $S$ is here a closed orientable surface of genus $g\geq 2$, and represent the Lorentzian analogue of quasi-Fuchsian hyperbolic manifolds. The analogy is enhanced by the fact that the moduli space of maximal globally hyperbolic metrics on $S\times\R$ (up to isometries isotopic to the identity) is parameterized by $\Teich(S)\times\Teich(S)$, where $\Teich(S)$ denotes the Teichm\"uller space of $S$ --- which is the analogue in this context of Bers' Simultaneous Uniformization Theorem.

Quasi-Fuchsian manifolds have been widely studied, \cite{MR2144972,MR847953,huangwang,MR2407332,MR2386723,MR795233,MR2172479,MR556580,seppiminimal}. In the celebrated paper \cite{BrockWPconvexcore}, Brock proved that the volume of the convex core of a quasi-Fuchsian manifold $M$ behaves coarsely like the Weil-Petersson distance between the two components in $\Teich(S)\times\Teich(S)$ provided by Bers' parameterization (\cite{MR0111834}). The main purpose of this paper is to study how the volume of the convex core of maximal globally hyperbolic manifolds is related to some analytic or geometric quantities only depending on the two parameters in Teichm\"uller space.



\subsection*{$L^1$-energies and length of earthquake laminations}

By an abuse of notation, we will use $h$ to denote the class of a hyperbolic metric in $\Teich(S)$, and $M_{h,h'}$ will denote the maximal globally hyperbolic manifold corresponding to the point $(h,h')\in\Teich(S)\times\Teich(S)$ in Mess' parameterization.
The main result of this paper is the fact that the volume of a maximal globally hyperbolic Anti-de Sitter manifold  roughly behaves like the minima of certain types of $L^1$-energies of maps $f:(S,h)\to(S,h')$. In fact, in his groundbreaking preprint \cite{Thurstondistance} about the Lipschitz asymmetric distance, Thurston suggested the interest in studying other type of $L^p$-energies, in contrast to the case $p=\infty$ corresponding to the Lipschitz distance. In this paper, we will consider the functional which corresponds to $p=1$:
$$C^1_{\mathrm{id}}(S)\ni f\mapsto\int_S ||d f||d\mathrm A_h~,$$
where  $C^1_{\mathrm{id}}(S)$ denotes the space of $C^1$ self-maps of $S$ homotopic to the identity, and
 $||df||$ is the norm of the differential of $f$, computed with respect to the metrics $h$ and $h'$ on $S$. (See Definition \ref{defi energial1}.) This functional is usually called $L^1$-\emph{energy}, or \emph{total variation}, as it coincides with the total variation in the sense of BV maps. Our main result is the following:

\begin{thmx} \label{cor energiaL1}
Let $M_{h,h'}$ be a maximal globally hyperbolic  $\AdS^{3}$ manifold. {Then
$$\frac{1}{4}\inf_{f\in C^1_{\mathrm{id}}(S)} \int_S ||d f||d\mathrm A_h -\frac{\sqrt{2}}{2}\pi|\chi(S)|\leq \Vol(\mathcal{C}(M_{h,h'}))\leq \frac{\pi^{2}}{2}|\chi(S)|+\frac{\sqrt{2}}{2}\inf_{f\in C^1_{\mathrm{id}}(S)} \int_S ||d f||d\mathrm A_h~.$$}
\end{thmx}

Observe that the volume of the convex core vanishes precisely for Fuchsian manifolds, that is, on those manifolds containing a totally geodesic spacelike surface. Those manifolds correspond to the diagonal in $\Teich(S)\times\Teich(S)$, that is, $h=h'$. In this case, a direct computation shows that the left-hand-side in the inequality of Theorem \ref{cor energiaL1} vanishes. Indeed, the $L^1$-energy is minimized by the identity map $f=\mathrm{id}:(S,h)\to(S,h)$, and for this map $||df||=\sqrt 2$ at every point.

Theorem \ref{cor energiaL1} will follow from two more precise statements about the behavior of the volume of the convex core, namely Theorem \ref{theorem hol energy} and Theorem \ref{prop:comparison volume lams} below. The former uses the relation of Anti-de Sitter geometry, and in particular maximal surfaces (i.e. with vanishing mean curvature), with minimal Lagrangian maps between hyperbolic surfaces. The latter relies instead on the connection between pleated surfaces and earthquake maps. \\

Let us consider first the 1-\emph{Schatten energy}. Given two hyperbolic surfaces $(S,h)$ and $(S,h')$, this functional, which we denote $E_{Sch}(\cdot,h,h')$, is defined as:
$$
C^1_{\mathrm{id}}(S)\ni f\mapsto \int_S \tr\left(\sqrt{df^*df}\right)d\mathrm A_h~,
$$
where $df^*$ is the $h$-adjoint operator of the differential $df$, and $\sqrt{df^*df}$ denotes the unique positive, symmetric square root of the operator ${df^*df}$. 
When $f$ is orientation-preserving, the functional $E_{Sch}(f,h,h')$ actually coincides with the \emph{holomorphic} $L^{1}$-\emph{energy}, which was already studied in \cite{oneharmonic}, and is defined (on the space $\Diffeo_{\mathrm{id}}(S)$ of diffeomorphisms isotopic to the identity) by: 
$$\Diffeo_{\mathrm{id}}(S)\ni f\mapsto\int_S ||\partial f||d\mathrm A_h~,$$
where $||\partial f||$ is the norm of the $(1,0)$-part of the differential of $f$. In \cite{oneharmonic}, Trapani and Valli proved that this functional admits a unique minimum, which coincides with the unique minimal Lagrangian diffeomorphism $m:(S,h)\to(S,h')$
isotopic to the identity (see also \cite{labourieCP} and \cite{Schoenharmonic}). Using the known construction (\cite{bon_schl}, \cite{Schlenker-Krasnov}) which associates a minimal Lagrangian diffeomorphism from $(S,h)$ to $(S,h')$, isotopic to the identity, to the unique maximal surface in $M_{h,h'}$, we obtain the following theorem which gives a precise description of the coarse behavior of the volume of the convex core in terms of the  1-Schatten energy. 
\begin{thmx}\label{theorem hol energy}
Let $M_{h,h'}$ be a maximal globally hyperbolic $\AdS^3$ manifold. Then
$$\frac{1}{4}E_{Sch}(m,h,h')-\pi|\chi(S)| \leq \Vol(\mathcal{C}(M_{h,h'}))\leq \frac{\pi^{2}}{2}|\chi(S)|+\frac{1}{4}E_{Sch}(m,h,h')~,$$
where $m:(S,h)\to (S,h')$ is the minimal Lagrangian map isotopic to the identity, that is, the minimum of the 1-Schatten energy functional $E_{Sch}(\cdot,h,h'):C^1_{\mathrm{id}}(S)\to\R$. 
\end{thmx}
Again, the left-hand-side vanishes precisely on Fuchsian manifolds, that is, precisely when $\Vol(\mathcal{C}(M_{h,h'}))$ vanishes as well. We remark that in the proof of Theorem \ref{theorem hol energy} we need the fact that the minimal Lagrangian map actually minimizes $E_{Sch}(\cdot,h,h')$ on $C^1_{\mathrm{id}}(S)$, which follows from the theorem of Trapani and Valli, and the convexity of the functional $E_{Sch}(\cdot,h,h')$.
The upper bound in Theorem \ref{cor energiaL1} then follows from Theorem \ref{theorem hol energy}, by using that, for any $f\in\Diffeo_{\mathrm{id}}(S)$ and every $x\in S$, $\tr(\sqrt{df^*df})\leq \sqrt 2||df_x||$.
\\

A more combinatorial version of the relation between maximal surfaces and minimal Lagrangian maps is the association, already discovered by Mess, of (left and right) earthquake maps from $(S,h)$ to $(S,h')$ from the two pleated surfaces which form the boundary of the convex core of $M_{h,h'}$. We recall that Thurston's Earthquake Theorem (\cite{thurstonearth}) asserts the existence of a unique left (and a unique right) earthquake map from $(S,h)$ to $(S,h')$, thus producing two \emph{measured geodesic laminations}. The \emph{length} of a measured geodesic lamination is then the unique continuous homogeneous function which extends the length of simple closed geodesics.


If we denote by $E^\lambda:\Teich(S)\to\Teich(S)$ the transformation which associates to $h\in\Teich(S)$ the metric $h'=E^\lambda(h)$ obtained by a (left or right) earthquake along $\lambda$, the following result gives a relation between the volume and the length of (any of the two) earthquake laminations:

\begin{thmx} \label{prop:comparison volume lams}
Given a maximal globally hyperbolic manifold $M_{h,h'}$, let $\lambda$ be the (left or right) earthquake lamination such that $E^\lambda(h)=h'$. Then
$$
\frac{1}{4}\ell_{\lambda}(h)\leq 
\Vol(\mathcal C(M_{h,h'}))\leq \frac{1}{4}\ell_{\lambda}(h)+\frac{\pi^2}{2}|\chi(S)|~.
$$
\end{thmx}

The lower inequality of Theorem \ref{cor energiaL1} is then a consequence of Theorem \ref{prop:comparison volume lams} and the basic observation that the total variation of the earthquake map along the lamination $\lambda$ is at most $\ell_\lambda(h)+2\sqrt{2}\pi|\chi(S)|$.

Another straightforward consequence of 
these results is the fact that the length of the left and right earthquake laminations, and the holomorphic energy of the minimal Lagrangian map (up to a factor), are comparable. Namely, their difference is bounded only in terms of the topology of $S$. For instance:

\begin{corx} \label{cor difference lengths}
Given two hyperbolic metrics $h$ and $h'$ on $S$, if $\lambda_l$ and $\lambda_r$ are the measured laminations such that $E_l^{\lambda_l}(h)=h'$ and $E_r^{\lambda_r}(h)=h'$, then
$$|\ell_{\lambda_l}(h)-\ell_{\lambda_r}(h)|\leq 2\pi^2|\chi(S)|~.$$
\end{corx}

More precisely, Corollary \ref{cor difference lengths} follows from a key proposition which is used throughout the paper, see Proposition \ref{prop maximal volume} below, and from an explicit formula for the volume of the union of the convex core of $M_{h,h'}$ and of one of the two ends, in terms of the earthquake laminations. (This was given in \cite[Section 8.2.3]{bebo}, see Lemma \ref{lemma volume and lams})
Corollary \ref{cor difference lengths} seems to be a non-trivial result to obtain using only techinques from hyperbolic geometry. 

\subsection*{Metrics on Teichm\"uller space}
A result like Brock's Theorem for maximal globally hyperbolic Anti-de Sitter manifolds, replacing Bers' parameterization by Mess' parameterization, turns out not to be true. The problem of relating the volume $\Vol(\mathcal C(M_{h,h'}))$ to the distance between $h$ and $h'$ for some metric structure on $\Teich(S)$ was mentioned in \cite[Question 4.1]{questionsads}. We show that the volume of the convex core of a maximal globally hyperbolic manifold $M=M_{h,h'}$ is bounded asymptotically from above by Thurston's asymmetric distance between $(S,h)$ and $(S,h')$ (actually by the minimum of the two asymmetric distances), from below by the Weil-Petersson distance. Neither of these bounds holds on both sides, hence this seems to be the best affirmative answer one can give to this question.
\\

Recall that Thurston's distance $d_\Th(h,h')$ is the logarithm of the best Lipschitz constant of diffeomorphisms from $(S,h)$ to $(S,h')$, isotopic to the identity. This definition satisfies the properties of a distance on $\Teich(S)$, except the symmetry. As the norm of the differential $||df||$, which appears in Theorem \ref{cor energiaL1}, is bounded pointwise by the Lipschitz constant of $f$, we derive the following 
bound from above of the volume with respect to the minimum of the two asymmetric distances:

\begin{thmx}\label{cor:confrontovolumeThurston}
Let $M_{h,h'}$ be a maximal globally hyperbolic $\AdS^{3}$ manifold. Then
\[
	\Vol(\mathcal{C}(M_{h,h'})) \leq \frac{\pi^{2}}{2}|\chi(S)|+\pi|\chi(S)|\exp(\min\{\dth(h,h'),\dth(h',h)\}) \ .
\]
\end{thmx}

{However, the volume of the convex core is not coarsely equivalent to the minimum of Thurston asymmetric distances, as we can produce examples of manifolds $M_{h_n,h'_n}$ in which the minimum $\min\{\dth(h_n,h_n'),\dth(h_n',h_n)\}$ goes to infinity while $\Vol(\mathcal{C}(M_{h_n,h_n'}))$ stays bounded, thus showing that there cannot be a bound from below on the volume using any of Thurston's asymmetric distances. However, in these examples the systole of both $h_{n}$ and $h_{n}'$ go to $0$, and this condition is necessary for this phenomenon to happen. More precisely, the volume $\Vol(\mathcal{C}(M_{h,h'}))$ is coarsely equivalent to the minimum of Thurston asymmetric distances if one of the two points $h,h'$ lie in the $\epsilon$-thick part of $\Teich(S)$.}

{We also provide a sequence of examples $M_{h_g,h'_g}$, on the surface $S_g$ of genus $g\geq 2$, in which the volume $\Vol(\mathcal{C}(M_{h_g,h_g'}))$ actually grows like $|\chi(S_g)|\exp(\min\{\dth(h_g,h_g'),\dth(h_g',h_g)\})$. Hence the growth, with respect to the genus, of the multiplicative factor in Theorem \ref{cor:confrontovolumeThurston} is basically optimal.}



\indent On the other hand, we obtain a coarse bound from below on the volume of the convex core of $M_{h,h'}$ by using the Weil-Petersson distance $d_\weil(h,h')$. 

\begin{thmx}\label{thm:stimaWP} 
Let $M_{h,h'}$ be a maximal globally hyperbolic $\AdS^{3}$ manifold. Then there exist some positive constants $a,b,c>0$  such that
\[
	\exp\left({\frac{a}{|\chi(S)|}\dwp(h,h')-b|\chi(S)|}\right)-c\leq \Vol(\mathcal C(M_{h,h'}))~.
\]
\end{thmx}

There are examples in which $d_\weil(h_n,h_n')$ remains bounded, but $\Vol(\mathcal{C}(M_{h_n,h_n'}))$ diverges, thus the volume of the convex core of $M_{h,h'}$ cannot be bounded from above by the Weil-Petersson distance between $h$ and $h'$.

\subsection*{Further techniques involved in the proofs}
Let us outline here some techniques of the proofs of the above results. A first main difference between the quasi-Fuchsian and the Anti-de Sitter setting consists in the fact that the volume of the whole maximal globally hyperbolic Anti-de Sitter manifold $M_{h,h'}$ is always finite. By considering the foliation by constant curvature surfaces (\cite{barbotzeghib}) of the complement of the convex core, we will show that the volume of $M_{h,h'}$ and the volume of its convex core are coarsely equivalent. More precisely we prove the following:

\begin{propx} \label{prop maximal volume}
Given a maximal globally hyperbolic manifold $M$, let $M_-$ and $M_+$ be the two connected components of the complement of $\mathcal C(M)$. Then
$$\Vol(M_-)\leq \frac{\pi^2}{2}|\chi(S)|\qquad\text{and}\qquad \Vol(M_+)\leq \frac{\pi^2}{2}|\chi(S)|~,$$
with equality if and only $M$ is Fuchsian.
\end{propx}

Using a foliation by equidistant surfaces from the boundary of the convex core, one can then prove the following formula (see also \cite{questionsads} and \cite[Subsection 8.2.3]{bebo}) which connects the volume of the convex core, the volume of $M_-$, and the length of the left earthquake lamination $\lambda$:
\begin{equation}\label{volume intro}
	\Vol(\mathcal C(M))+\Vol(M_-)=\frac{1}{4}\ell_{\lambda}(h)+\frac{\pi^2}{2}|\chi(S)| ~.
\end{equation}
Of course an analogous formula holds for the right earthquake and for the other connected component $M_+$.
Theorem \ref{prop:comparison volume lams} will then follow by combining Equation \eqref{volume intro} with Proposition \ref{prop maximal volume}. 
\\

\indent Another main consequence of Proposition \ref{prop maximal volume} is the fact that the volume of the convex core of $M_{h,h'}$ is coarsely equivalent to the volume of every domain in which it is contained. Starting from the unique maximal surface embedded in $M_{h,h'}$ (\cite{bbzads}), we construct a domain with smooth boundary $\Omega_{h,h'}$ which contains the convex core and whose volume can be computed explicitly in terms of the function (already introduced in \cite{bms2}) 
\begin{align*}
	F: \Teich(S)\times \Teich(S) &\rightarrow \R^{+}\\
		(h,h') &\mapsto \int_{S}\trace(b) d\mathrm{A}_{h}~,
\end{align*}
where $b=\sqrt{dm^*dm}\in\Gamma(\mathrm{End}(TS))$ is the unique Codazzi, $h$-self-adjoint operator such that $h'=h(b\cdot, b\cdot)$, associated to the minimal Lagrangian diffeomorphism $m:(S,h)\to(S,h')$. In particular, using explicit formulas that relate the embedding data of the maximal surface in $M_{h,h'}$ with the operator $b$, we can prove that 
\[
	\Vol(\Omega_{h,h'})=\frac{\pi^{2}}{2}|\chi(S)|+\frac{1}{4}\int_{S}\trace(b)d\mathrm{A}_{h} ~.
\]
Combined again with Proposition \ref{prop maximal volume}, this implies that the volume of the convex core is coarsely equivalent to the function $F$ defined above, and thus leads to Theorem \ref{theorem hol energy}.
In addition, since $\trace(b)$ is bounded from above by twice the Lipschitz constant of $m$, we will deduce Theorem \ref{cor:confrontovolumeThurston}. \\
\\
\indent As for the relation between the volume of the convex core of $M_{h,h'}$ and the Weil-Petersson distance between $h$ and $h'$, the main technical tool consists in the following estimate:

\begin{thmx}\label{thm:mainestimate} There exists a universal constant $a>0$ such that for every $\lambda \in \mathcal{ML}(S)$ and for every $h \in \Teich(S)$, we have
$$
	\| \grad \ell_{\lambda}(h)\|_{\weil}\geq \frac{a}{|\chi(S)|}\ell_{\lambda}(h) \ .
$$
\end{thmx}

This will be obtained by a careful analysis of Riera's formula \cite{riera} for the norm of the Weil-Petersson gradient of the length function. With this in hand, the proof of Theorem \ref{thm:stimaWP} then goes as follows. By a result of Bers', we can fix a pants decomposition $P$ for $h$ such that the length of all curves $\alpha_{j}$ in $P$ are smaller than a constant $L>0$ (only depending on the genus of $S$). If the metric $h'$ is obtained by performing a (left or right) earthquake along a measured geodesic lamination $\lambda$, then the length of the curves $\alpha_{j}$ increases at most by
\[
	\ell_{\alpha_{j}}(h') \leq L+\frac{\ell_{\lambda}(h)}{d(L)} \ ,
\]
where $d(L)$ is a constant depending only on $L$. We can thus say that the point $h' \in \Teich(S)$ belongs to the set
\[
	V_{m}(S)=\{ h \in \Teich(S) \ | \ \ell_{\alpha_{j}}(h)<m\}
\]
if we put $m=L+\ell_{\lambda}(h)/d(L)$. As a consequence of Theorem \ref{thm:mainestimate}, the integral curve of the vector field $X=-\grad \ell_{\lambda_{P}}/\|\grad \ell_{\lambda_{P}}\|_{\weil}$, (where we denoted with $\lambda_{P}$ the measured geodesic lamination consisting of the simple closed curves $\alpha_{j}$ with unit weight) starting at $h'$ will intersect the set $V_{L}(S)$ in a finite time $t_{0}$, which we are able to express explicitely in terms of $\ell_{\lambda}(h)$ and the constants $L$ and $d(L)$. Theorem \ref{thm:stimaWP} will then follow from the fact that the set $V_{L}(S)$ has bounded diameter for the Weil-Petersson metric and from Theorem \ref{prop:comparison volume lams}.

\subsection*{Organization of the paper}In Section \ref{sec:preliminari} we introduce maximal globally hyperbolic Anti-de Sitter manifolds and we underline their relation with earthquake maps and minimal Lagrangian diffeomorphisms. In Section \ref{sec:lunghezzalaminazioni} we study the volume of the complement of the convex core and we prove Theorem \ref{prop:comparison volume lams} and Corollary \ref{cor difference lengths}. Section \ref{sec:energiaL1} is devoted to the comparison between the volume of the convex core and the holomorphic energy (Theorem \ref{theorem hol energy}). Then Theorem \ref{cor energiaL1}  is proved in Section \ref{sec main thm}. In Section \ref{sec:distanzaThurston} we study the relation between the volume of the convex core and Thurson's asymmetric distance, thus proving Theorem \ref{cor:confrontovolumeThurston}. In Section \ref{sec:distanzaWP} we focus on the lower bound with respect to the Weil-Petersson metric, in particular Theorem \ref{thm:stimaWP}. The main estimate of Theorem \ref{thm:mainestimate} is then proved in Section \ref{sec:stimagradiente}. 

\subsection*{Acknowledgements}
A large part of this work was done during a visit of the second author to the University of Luxembourg, and a visit of the third author to the University of Pavia. We are grateful to those Institutions for their hospitality. We would like to thank Olivier Glorieux and Nicolas Tholozan for several interesting discussions. The third author would also like to thank his PhD advisor Jean-Marc Schlenker for many useful conversations about Anti-de Sitter geometry. We would like to thank an anonymous referee for several useful comments and remarks.

\section{Preliminaries on Anti-de Sitter geometry}\label{sec:preliminari}

Anti-de Sitter space $\AdS^3$ is a Lorentzian manifold of signature $(2,1)$, topologically a solid torus, of constant sectional curvature $-1$. We will adopt the following $\PSL(2,\R)$ model of $\AdS^3$: consider the quadratic form 
$$q=-\det$$ 
on the vector space $\mathfrak{gl}(2,\R)$ of 2-by-2 matrices. It is easy to check that $q$ is non-degenerate and has signature $(2,2)$. Hence the restriction of $q$ induces a quadratic form of signature $(2,1)$  on the submanifold of $\mathfrak{gl}(2,\R)$ defined by the equation $q=-1$. Namely, on
$\SL(2,\R)$, the submanifold of square matrices with unit determinant. 

Therefore, the polarization of $q$ defines a Lorentzian metric on $\SL(2,\R)$. This Lorentzian metric is invariant by multiplication by $-\mathrm{id}$, and therefore it defines a Lorentzian metric on
$$\PSL(2,\R)=\SL(2,\R)/\{\pm\mathrm{id}\}~.$$
The metric on $\PSL(2,\R)$ will be simply denoted by $\langle\cdot,\cdot\rangle$. 

\begin{defi}
Anti-de Sitter space of dimension three is $\PSL(2,\R)$ endowed with the Lorentzian metric: 
$$\AdS^3=(\PSL(2,\R),\langle\cdot,\cdot\rangle)~.$$
\end{defi}

Anti-de Sitter space is orientable and time-orientable, and its group of orientation and time-orientation preserving isometries is:
\begin{equation} \label{eq isometry group}
\isom_0(\AdS^3)=\PSL(2,\R)\times \PSL(2,\R)~,
\end{equation}
where the action of $\PSL(2,\R)\times \PSL(2,\R)$ is the following:
$$(\alpha,\beta)\cdot \gamma=\alpha\cdot\gamma\cdot \beta^{-1}~,$$
for any $\alpha,\beta,\gamma\in \PSL(2,\R)$.

The \emph{boundary} of $\AdS^3$ is defined as the projectivization of rank 1 matrices:
$$\partial\AdS^3=\mathbb{P}\{\alpha\neq 0:\det\alpha=0\}$$
where the convergence of a sequence of points in $\AdS^3$ to a point in $\partial\AdS^3$ is defined by considering $\AdS^3$ and $\partial\AdS^3$ as subsets of $\R\mathbb{P}^3$. Then 
\begin{equation} \label{eq boundary}
\partial\AdS^3\cong \mathbb{RP}^1\times \mathbb{RP}^1~,
\end{equation}
where the identification sends $[M] \in \partial\AdS^{3}$ to the couple $([\Ima(M)], [\Ker(M)])$. It can then be checked that the action of $\isom_0(\AdS^3)$ extends to $\partial\AdS^3$ and corresponds (using Equations \eqref{eq isometry group} and \eqref{eq boundary}) to the obvious action of $\PSL(2,\R)\times \PSL(2,\R)$ on $\mathbb{RP}^1\times \mathbb{RP}^1$.

\subsection{Maximal globally hyperbolic Anti-de Sitter manifolds} \label{sec:volume}

Let us recall the definition of (maximal) globally hyperbolic manifolds:

\begin{defi} \label{defi MGH}
A Lorentzian manifold $M$ is \emph{globally hyperbolic} if is contains a \emph{Cauchy surface}, that is, a spacelike surface $S$ which intersects every inextensible causal curve exactly in one point. Such a manifold $M$ is \emph{maximal} if every isometric embedding $\varphi:M\to M'$, such that $M'$ is a globally hyperbolic manifold and $\varphi(S)$ is a Cauchy surface of $M'$, is necessarily a global isometry.
\end{defi}

Any globally hyperbolic manifold $M$ as in Definition \ref{defi MGH} is topologically a product, $M\cong S\times \R$, see \cite{MR0270697}. We will be interested in Anti-de Sitter maximal globally hyperbolic manifolds, namely, locally isometric to $\AdS^3$. Moreover, in this paper any Cauchy surface will have the topology of a closed, oriented surface $S$ of genus $g\geq 2$.

The deformation space of those objects was described by Mess in \cite{Mess}. More precisely, let us denote
$$\mathcal{MGH}(S)=\{\text{maximal globally hyperbolic }\AdS^3\text{ metrics on }S\times\R\}/\Diffeo_{\mathrm{id}}(S\times\R)~,$$
where $\Diffeo_{\mathrm{id}}$ denotes the group of diffeomorphisms isotopic to the identity, and it acts on the set of metrics by pull-back. This definition actually resembles the following definition of Teichm\"uller space:
$$\Teich(S)=\{\text{hyperbolic metrics on }S\}/\Diffeo_{\mathrm{id}}(S)~.$$
(In this last definition, a hyperbolic metric simply is a Riemannian metric of constant curvature $-1$.) Then Mess' result is the following description of $\mathcal{MGH}(S)$ in terms of Teichm\"uller space:

\begin{theorem}[\cite{Mess}] \label{mess theorem}
For every closed oriented surface $S$ of genus $g\geq 2$,
$$\mathcal{MGH}(S)\cong \Teich(S)\times \Teich(S)~.$$
\end{theorem}

The homeomorphism of Theorem \ref{mess theorem} goes as follows. First, recall that the Teichm\"uller space of $S$ is identified to a certain connected component in the space of representations of $\pi_1(S)$ into $\PSL(2,\R)$, considered up to conjugation. In fact, this identification is obtained by taking the conjugacy class of the holonomy represention of a hyperbolic metric on $S$, and the desired connected component is given by the subset of representations with \emph{maximal}  Euler class:
$$\Teich(S)\cong\{\rho_0:\pi_1(S)\to\PSL(2,\R): \ e(\rho_0)=|\chi(S)|\}/\PSL(2,\R)~.$$

See \cite{goldmanthesis}. Representations of maximal Euler class, that is, such that $e(\rho_0)=|\chi(S)|$, are called \emph{Fuchsian}.

Mess proved that for every globally hyperbolic $\AdS^3$ manifold $M$, the holonomy representation
$$\rho=(\rho_l,\rho_r):\pi_1(S)\to \PSL(2,\R)\times \PSL(2,\R)$$
is such that $e(\rho_l)=e(\rho_r)=|\chi(S)|$, and therefore $([\rho_l],[\rho_r])$ defines a point in $\Teich(S)\times \Teich(S)$.
(Here $\rho:\pi_1(S)\to\isom_0(\AdS^3)$ is identified to a pair of representations into $\PSL(2,\R)$ by means of Equation \eqref{eq isometry group}.)
The representations $\rho_l$ and $\rho_r$ are called \emph{left holonomy} and \emph{right holonomy}.

\begin{example} \label{ex fuchsian}
If $h$ is a hyperbolic metric on $S$, then one can define the following metric on $M=S\times (-\pi/2,\pi/2)$, where $t$ is the ``vertical'' coordinate:
\begin{equation} \label{eq fuchsian metric}
g_h=-dt^2+\cos^2(t)h~.
\end{equation}
It turns out that $g_h$ has constant sectional curvature $-1$, that $S\times\{0\}$ is a totally geodesic Cauchy surface, and that $(M,g_h)$ is maximal globally hyperbolic. A developing map 
$$\mathrm{dev}:\tilde M=\tilde S\times (-\pi/2,\pi/2)\to\AdS^3$$ 
can be chosen so that the lift of the Cauchy surface $S\times\{0\}$ is mapped to the totally geodesic plane 
\begin{equation} \label{defi plane}
P_0=\{\alpha\in \PSL(2,\R):\alpha^2=\mathrm{id}\}=\{\alpha\in \PSL(2,\R):\trace(\alpha)=0\}~.
\end{equation}
Moreover, the lines $\{\star\}\times (-\pi/2,\pi/2)$ in $\tilde M$ are sent to the timelike geodesics $l$ orthogonal to $P_0$, where $t$ corresponds to the arclenght parameter. It turns out that every such geodesic $l$ coincides with the set $l_x$ of elliptic elements of $\PSL(2,\R)$ which fix a point $x$ in the upper-half plane. All these geodesics $l_x$ meet, for $t=\pm\pi/2$, at $[\mathrm{id}]\in\PSL(2,\R)$. 

Since the left and right holonomy of $(M,g_h)$ preserve the plane $P_0$, they also preserve $[\mathrm{id}]\in\PSL(2,\R)$, and therefore $\rho_l(\gamma)=\rho_r(\gamma)$. The maximal globally hyperbolic manifolds for which $[\rho_l]=[\rho_r]\in\Teich(S)$ are called \emph{Fuchsian} and correspond to the diagonal in 
$$\mathcal{MGH}(S)\cong \Teich(S)\times \Teich(S)~.$$
Equivalently, they contain a totally geodesic spacelike surface isometric to $\Hyp^2/\rho_0(\pi_1(S))$, where $\rho_0:=\rho_l=\rho_r$.

\end{example}

The following is a direct computation of the volume, using Equation \eqref{eq fuchsian metric}:

\begin{cor} \label{cor volume fuchsian}
The volume of any Fuchsian globally hyperbolic $\AdS^3$ manifold $M_F$ homeomorphic to $S\times\R$ is:
\begin{equation} \label{eq volume fuchsian metric}
\Vol(M_F)=\pi^2|\chi(S)|~.
\end{equation}
\end{cor}

\begin{remark} \label{rmk dual}
From Example \ref{ex fuchsian}, it turns out that the timelike geodesics orthogonal to any totally geodesic plane $P$
all meet at the same point at a distance $\pi/2$ in both directions. Such point of intersection will be called \emph{dual point}, since it can be interpreted also as the dual of $P$ with respect to the duality of $\mathbb{RP}^3$. 

It turns out that two spacelike planes $P_1$ and $P_2$ intersect if and only if their dual points $p_1$ and $p_2$ are connected by a spacelike segment. In this case the length of the spacelike segment connecting $p_1$ and $p_2$ equals the (hyperbolic) angle between the unit normal vectors of $P_1$ and $P_2$.
\end{remark}

Going back to Theorem \ref{mess theorem}, Mess explicitly constructed an inverse of the map $\mathcal{MGH}(S)\to \Teich(S)\times \Teich(S)$ we have just defined. For this purpose, given a pair of representations $\rho_l,\rho_r:\pi_1(S)\to\PSL(2,\R)$, it is well-known that there exists a unique 
orientation-preserving homeomorphism $\phi:\mathbb{RP}^1\to \mathbb{RP}^1$ such that $\phi\circ\rho_l(\gamma)=\rho_r(\gamma)\circ \phi$. By the identification of Equation \eqref{eq boundary}, the graph of $\phi$ corresponds to a submanifold of $\partial\AdS^3$, and
Mess showed that $\rho(\pi_1(S))$ acts freely and properly discontinously on the \emph{domain of dependence} of $gr(\phi)$ (see definition below). The quotient of the domain of dependence $\mathcal{D}(\phi)$ is a maximal globally hyperbolic $\AdS^3$ manifold, whose left and right holonomies are $\rho_l$ and $\rho_r$ by construction.

\begin{defi}
Given an orientation-preserving homeomorphism $\phi:\mathbb{RP}^1\to \mathbb{RP}^1$, 
the \emph{domain of dependence} of $gr(\phi)$, which we denote by $\mathcal{D}(\phi)$, is the union of points $p\in\AdS^3$ which are dual to spacelike planes $P$ such that $\partial P$ is disjoint from $gr(\phi)$ in $\partial\AdS^3$.
\end{defi}

If $\rho_l,\rho_r:\pi_1(S)\to\PSL(2,\R)$ are Fuchsian representations such that
 $\phi\circ\rho_l(\gamma)=\rho_r(\gamma)\circ \phi$ for every $\gamma\in\pi_1(S)$, then we denote
 $$M_{h_l,h_r}
 :=\mathcal{D}(\phi)/(\rho_l,\rho_r)(\pi_1(S))~,$$
 where $h_l$ and $h_r$ are the hyperbolic metrics of $S$ induced by $\Hyp^2/\rho_l(\pi_1(S))$ and $\Hyp^2/\rho_r(\pi_1(S))$ respectively. It follows from Mess' proof that the class of $M_{h_l,h_r}$ in $\mathcal{MGH}(S)$
only depends on the class of $h_l$ and $h_r$ in $\Teich(S)$.

Mess introduced the notion of convex core in maximal globally hyperbolic $\AdS^3$ manifolds.

\begin{defi} \label{defi convex core}
Given $\rho_l,\rho_r:\pi_1(S)\to\PSL(2,\R)$ Fuchsian representations, let $\phi:\mathbb{RP}^1\to \mathbb{RP}^1$ 
be the orientation-preserving homeomorphism such that $\phi\circ\rho_l(\gamma)=\rho_r(\gamma)\circ \phi$ for every $\gamma\in\pi_1(S)$, and let $\mathcal{C}(\phi)$ be the \emph{convex hull} of $gr(\phi)$. 

Then 
 $$\mathcal{C}(M_{h_l,h_r}):=\mathcal{C}(\phi)/(\rho_l,\rho_r)(\pi_1(S))$$
  is the \emph{convex core} of the maximal globally hyperbolic manifold $M_{h_l,h_r}$, where $h_l$ and $h_r$ are the hyperbolic metrics of $S$ induced by $\Hyp^2/\rho_l(\pi_1(S))$ and $\Hyp^2/\rho_r(\pi_1(S))$.
\end{defi}

There is a non-trivial point in Definition \ref{defi convex core}, that is, the convex hull of $gr(\phi)\subset \partial\AdS^3$ is contained in $\AdS^3$. This is not obvious since $\AdS^3$ is not convex in $\mathbb{RP}^3$. Moreover, the definition makes sense since the convex hull $\mathcal{C}(\phi)$ is contained in the domain of dependence $\mathcal{D}(\phi)$.
 See \cite{bon_schl}.

\subsection{Earthquake maps from pleated surfaces}

It is well-known that the convex hull of a maximal globally hyperbolic manifold $M$ is topologically $S\times I$, for $I$ a closed interval, and its two boundary components $\partial_\pm\mathcal C(M)$ are \emph{pleated surfaces}, that is, 
there exists an isometric map from a hyperbolic surface $(S,h_\pm)$ to $M$, with image $\partial_\pm\mathcal C(M)$, 
such that every point $x\in S$ 
is in the interior of an $h_\pm$-geodesic arc which is mapped isometrically to a spacelike geodesic of $M$.

This defines a \emph{bending lamination} $\lambda_\pm$ on $\partial_\pm\mathcal C(M)$, and a transverse measure is defined, which encodes the amount of bending and makes $\lambda_\pm$ a measured geodesic lamination, \cite{bebo}. Let  $\phi$ be the homeomorphism conjugating 
$\rho_l$ to $\rho_r$, as above. One can moreover define maps
$$\pi_l^\pm:\partial_\pm\mathcal C(\phi) \to\Hyp^2\qquad \pi_r^\pm:\partial_\pm\mathcal C(\phi) \to\Hyp^2$$
in the following way. 
Given any point $x\in \partial_+ \mathcal C(\phi)$ which admits a unique totally geodesic support plane $P_x$, consider the unique left isometry $\alpha_x$ which maps $P_x$ to $P_0$. Recall that $P_0$, defined in Equation \eqref{defi plane}, is the totally geodesic plane composed of the involutions of $\PSL(2,\R)$, and
 is identified to $\Hyp^2$ by means of the map
\begin{equation} \label{eq identification}
\alpha\in P_0\mapsto \mathrm{Fix}(\alpha)\in\Hyp^2~.
\end{equation}
 Then define
$$\pi_l^+:x\mapsto\alpha_x(x)\mapsto \mathrm{Fix}(\alpha_x(x))\in\Hyp^2\,,$$
Analogous definitions hold for $\pi_r^+$, using right isometries, and for $\pi_l^-$ and $\pi_r^-$, using the other boundary $\partial_- \mathcal C(\phi)$ of the convex hull.
The map $\pi_l^+$ induces an earthquake map from $(S,h_+)$ to $(S,h_l)$, defined in the complement of the weighted leaves of the bending lamination $\lambda_+$,
where we recall that $h_l$ is the left hyperbolic metric given by $\Hyp^2/(\rho_l(\pi_1(S))$ and $h_+$ is the metric induced on $\partial_+\mathcal C(M)$. Analogously there is a right earthquake map from $(S,h_+)$ to $(S,h_r)$
and similar maps are obtained from the lower boundary of $\mathcal C(M)$.

Given a measured geodesic lamination $\lambda$, we will denote 

$$E_l^{\lambda}:\Teich(S)\to\Teich(S)~,$$
the left earthquake map, and
$$E_r^{\lambda}:\Teich(S)\to\Teich(S)~,$$
the right earthquake map.
Hence, in \cite{Mess}, a new interpretation in terms of Anti-de Sitter geometry was given to the proof in \cite{thurstonearth} of the following celebrated theorem of Thurston:

\begin{theorem}[Earthquake theorem] \label{thm thurston}
Given two hyperbolic metrics $h,h'$ on a closed oriented surface $S$, there exists a unique pair of measured laminations $\lambda_l$, $\lambda_r$ such that
$$E_l^{\lambda_l}(h)=h' \qquad \text{and}\qquad E_r^{\lambda_r}(h)=h'~.$$ 
\end{theorem}

\subsection{Minimal Lagrangian maps from maximal surfaces} \label{subsec min lag maximal surf}

By a similar construction, it is possible to use Anti-de Sitter geometry to construct minimal Lagrangian maps between hyperbolic surfaces. Let us recall the definition of minimal Lagrangian diffeomorphism:

\begin{defi} \label{defi min lag} An orientation-preserving diffeomorphism $m:(S,h) \rightarrow (S,h')$ is minimal Lagrangian if it is area-preserving and its graph is a minimal surface in $(S\times S, h\oplus h')$.
\end{defi}

It is known \cite[Proposition 1.3]{bon_schl} that minimal Lagrangian diffeomorphisms are characterized by having a decomposition $m=(f')\circ f^{-1}$, where $f$ and $f'$ are harmonic maps with respect to a complex structure $X$ on $S$ and have opposite Hopf differential. That is:
$$f:(S,X)\to (S,h)\qquad\text{and}\qquad f':(S,X)\to (S,h')$$
are harmonic, where $(S,X)$ is Riemann surface, and 
$$\mathrm{Hopf}(f)=-\mathrm{Hopf}(f')~.$$

Given a maximal surface $\Sigma_0$ in a maximal globally hyperbolic $\AdS^3$-manifold, that is, a space-like surface of vanishing mean curvature, consider its lift to $\AdS^3$, which is a surface $\widetilde\Sigma_0$ of vanishing mean curvature with $\partial \widetilde\Sigma_0=gr(\phi)$, where $\phi:\mathbb{RP}^1\to \mathbb{RP}^1$ 
 the orientation-preserving homeomorphism such that $\phi\circ\rho_l(\gamma)=\rho_r(\gamma)\circ \phi$ for every $\gamma\in\pi_1(S)$. Then one can define the two projections
$$\pi_l^0:\widetilde\Sigma_0 \to\Hyp^2\qquad \pi_r^0:\widetilde\Sigma_0 \to\Hyp^2$$
exactly in the same way as we did in the previous section, that is,
$$\pi_l^0:x\mapsto\alpha_x(x)\mapsto \mathrm{Fix}(\alpha_x(x))\in\Hyp^2\,,$$
where $\alpha_x$ is the unique left isometry sending the plane $P_x$ tangent to $\widetilde\Sigma_0$ at $x$ to the plane $P_0$ of involutions, and we are identifying $P_0$ to $\Hyp^2$ by means of \eqref{eq identification}. An analogous definition holds for $\pi_r^0$.

It turns out that 
$\pi_l^0$ and $\pi_r^0$ are harmonic diffeomorphisms, thus they induce harmonic maps from $\Sigma_0$ to $(S,h)$ and $(S,h')$, which have opposite Hopf differential. Hence the map $m$ {associated} to the maximal surface $\Sigma_0$ is a minimal Lagrangian diffeomorphism. Moreover, all minimal Lagrangian diffeomorphisms from $(S,h)$ to $(S,h')$ are obtained in this way. Hence, from results (see for instance \cite{bbzads} and \cite{bon_schl} for a generalisation) of existence and uniqueness of maximal surfaces in maximal globally hyperbolic $\AdS^3$-manifolds 
one can reprove the following theorem (see \cite{labourieCP}, \cite{Schoenharmonic}):

\begin{theorem} \label{thm min lag}
Given two hyperbolic metrics $h,h'$ on a closed oriented surface $S$, there exists a unique minimal Lagrangian diffeomorphism $m:(S,h)\to (S,h')$ isotopic to the identity.
\end{theorem}

\section{Length of earthquake laminations}\label{sec:lunghezzalaminazioni}

We will discuss an explicit relation {between} the volume of a maximal globally hyperbolic manifold (or the volume of its convex core) {and} the length of the (left and right) earthquake laminations. 

Before that, we will prove that the volume of the complement of the convex hull is bounded by the volume of a Fuchsian manifold. That is, the volume of $M\setminus \mathcal C(M)$ is maximal in the Fuchsian case. Hence, from a coarse point of view, the volume of the manifold $M$ is essentially the same as the volume of its convex core.

\subsection{Volume of the complement of the convex hull.}

Given a maximal globally hyperbolic manifold $M$, we will denote
$$M\setminus\mathcal C(M)=M_+\sqcup M_-~,$$
where $M_+$ is the connected component adjacent to $\partial_+\mathcal C(M)$, and $M_-$ the other connected component. The following proposition estimates the volume of the complement of the convex core.

\begin{repprop}{prop maximal volume}
Given a maximal globally hyperbolic manifold $M$, let $M_-$ and $M_+$ be the two connected components of the complement of $\mathcal C(M)$. Then
$$\Vol(M_-)\leq \frac{\pi^2}{2}|\chi(S)|\qquad\text{and}\qquad \Vol(M_+)\leq \frac{\pi^2}{2}|\chi(S)|~,$$
with equality if and only $M$ is Fuchsian.
\end{repprop}
It will then obviously follow that 
\begin{equation} \label{eq pi2}
\Vol(M\setminus \mathcal C(M))\leq \pi^2|\chi(S)|~,
\end{equation}
that is (compare Corollary \ref{cor volume fuchsian}), the volume of $\Vol(M\setminus \mathcal C(M))$ is at most the volume of a Fuchsian manifold.

\begin{proof}
We will give the proof for $M_{+}$. By \cite{barbotzeghib}, there exists a function
$$F:M_+\to [0,\infty)$$
such that $S_\kappa=F^{-1}(\kappa)$ is the surface of constant curvature $K=-1-\kappa$. By the Gauss equation in the $\AdS^3$ setting, if $B$ is the shape operator of $S_\kappa$, then $\kappa=\det B_x$ for every point $x\in S_\kappa$.

Let $\varphi_{t}$ be the flow of the vector field $\grad F/||\grad F||^2$. By definition $\varphi_{t}(S_{\kappa})=S_{\kappa +t}$ and following this flow, one obtains the following expression for the volume of $M_+$:
\begin{equation} \label{eq volume kappa}
\Vol(M_+)=\int_0^\infty\int_{S_\kappa}\frac{d\mathrm{Area}_{S_\kappa}}{||\grad F||}d\kappa~.
\end{equation}

Let us now fix some $\kappa\in(0,\infty)$. Let $x(\rho)$ be the normal flow of $S_{\kappa}$, which is well-defined for a small $\rho$. We adopt the convention that the unit normal of $S_{\kappa}$ is pointing towards the concave side of $S_{\kappa}$. Let $S_{\kappa}(\rho)$ be the parallel surface of $S_{\kappa}$ at distance $\rho$, in the concave side. Using the formula for the shape operator of $S_{\kappa}(\rho)$, see \cite{Schlenker-Krasnov} or \cite[Lemma 1.14]{seppimaximal}, we get:
$$B_\rho=(\cos(\rho) E+\sin(\rho) B)^{-1}(-\sin(\rho) E+\cos(\rho) B)$$
where $B_\rho$ is the shape operator of $S_{\kappa}(\rho)$.
Hence
$$\det B_\rho=\frac{\sin^2\rho+\cos^2\rho\det B-(\sin\rho\cos\rho)\tr B}{\cos^2\rho+\sin^2\rho\det B+(\sin\rho\cos\rho)\tr B}~.$$
With our convention, $\tr B<0$ since $S_\kappa$ is concave, and thus, using the inequality $(\tr B)^2\geq 4\det B=4\kappa$, we have
$$\det B_\rho\geq \frac{\sin^2\rho+\cos^2\rho\,\kappa+2\sin\rho\cos\rho\sqrt\kappa}{\cos^2\rho+\sin^2\rho \,\kappa-2\sin\rho\cos\rho\sqrt\kappa}~,$$
which implies that $\inf \det B_\rho\geq f(\rho)$,
where
$$f(\rho)=\frac{\sin^2\rho+\cos^2\rho\,\kappa+2\sin\rho\cos\rho\sqrt\kappa}{\cos^2\rho+\sin^2\rho \,\kappa-2\sin\rho\cos\rho\sqrt\kappa}~.$$
Hence by an application of the maximum principle (compare for instance \cite{barbotzeghib}) one has that $S_{\kappa}(\rho)$ lies entirely in the concave side of $S_{f(\rho)}$. In other words, $F(x(\rho))\geq f(\rho)$.

Observe that the timelike vector fields $\dot x(\varrho)$ and $\grad F(x(\varrho))$ are collinear when $\varrho=0$. Hence for every $\epsilon>0$, there exists $\rho_0>0$ such that for every $\varrho<\rho_0$ we have
\[
	\langle\grad F(x(\varrho)),\dot x(\varrho)\rangle \leq ||\grad F(x(\varrho))||(1+\epsilon) \ ,
\]
hence
$$
F(x(\rho))-F(x(0))= \int_0^\rho\langle\grad F(x(\varrho)),\dot x(\varrho)\rangle d\varrho \\
\leq (1+\epsilon) \int_0^\rho ||\grad F(x(\varrho))||d\varrho~,
$$
for $\rho<\rho_0$. On the other hand
$$F(x(\rho))-F(x(0))\geq f(\rho)-\kappa~,$$
and thus by differentiating at $\rho=0$:
$$||\grad F(x(0))||(1+\epsilon)\geq \left.\frac{d}{d\rho}\right|_{\rho=0}f(\rho)=2\sqrt\kappa(\kappa+1)~.$$

We can finally conclude the computation. From Equation \eqref{eq volume kappa}, we have
$$\Vol(M_+)\leq \int_0^{+\infty} \frac{(1+\epsilon)}{2\sqrt\kappa(\kappa+1)}\int_{S_\kappa}{d\mathrm{Area}_{S_\kappa}}d\kappa=(1+\epsilon)\int_0^{+\infty} \frac{2\pi|\chi(S)|d\kappa}{2\sqrt\kappa(\kappa+1)^2}=\frac{\pi^2}{2}|\chi(S)|(1+\epsilon)~,$$
where we have used the Gauss-Bonnet formula, the fact that the Gaussian curvature of $S_\kappa$ is $-1-\kappa$, and that
$$\int_0^{+\infty} \frac{dx}{\sqrt{x}(1+x)^2}=\frac{\pi}{2}~.$$ 

Finally, let us observe that equality holds if and only if $(\tr B)^2=4\det B$ at every point, which is the case in which all the surfaces $S_\kappa$ are umbilical. This implies that the boundary of the convex core is totally geodesic, and thus $M$ is a Fuchsian manifold.
\end{proof}

A direct consequence of Proposition \ref{prop maximal volume}, using Equation \eqref{eq pi2}, is that the volume of $M$ and of the convex core of $M$ are roughly comparable:

\begin{cor} \label{cor volume manifold and convex core}
Given a maximal globally hyperbolic manifold $M$,
$$\Vol(\mathcal C(M))\leq \Vol(M)\leq \Vol(\mathcal C(M))+\pi^2|\chi(S)|~.$$
\end{cor}

%

\subsection{Length of earthquake laminations}

In this subsection we will prove a coarse relation between the volume of a maximal globally hyperbolic manifold $M_{h,h'}$ and the length of the earthquake laminations of the (both left and right) earthquake maps from $(S,h)$ to $(S,h')$, provided by the Earthquake Theorem ({Theorem} \ref{thm thurston}). 

Before stating the main results of this subsection, we finally need to recall the definition of \emph{length} of a measured geodesic lamination. Let us denote by
$\mathcal{ML}(S)$ the set of measured laminations on $S$, up to isotopy.
 The set
of weighted multicurves 
$$(\mathbf{c},\mathbf{a})=((c_1,a_1),\ldots,(c_n,a_n))~,$$
 where $c_i$ are essential simple closed curves on $S$ and $a_i$ are positive weights, is dense in $\mathcal{ML}(S)$. The well-posedness of the following definition then follows from \cite{MR847953}.

\begin{defi}
Given a closed orientable surface $S$ of genus $g\geq 2$, we denote
$$\ell: \mathcal{ML}(S)\times \Teich(S)\to [0,+\infty)$$
the unique continuous function such that, for every weighted multicurve $(\mathbf{c},\mathbf{a})$,
$$\ell((\mathbf{c},\mathbf{a}),[h])=\sum_{i=0}^n a_i\mathrm{length}_{h}(c_i)~,$$
where $\mathrm{length}_{h}(c)$ denotes the length of the $h$-geodesic representative 
in the isotopy class of $c$.
Then we define the \emph{length function} associated to a measured lamination $\lambda$ as the function
$$\ell_{\lambda}:\Teich(S)\to[0,+\infty)$$
defined by $\ell_{\lambda}([h])=\ell(\lambda,[h])$.
\end{defi}

Similarly, we also recall the definition of \emph{topological intersection} for measured geodesic laminations:
\begin{defi} \label{defi intersection}
Given a closed orientable surface $S$ of genus $g\geq 2$, we denote
$$\iota: \mathcal{ML}(S)\times \mathcal{ML}(S)\to [0,+\infty)$$
the unique continuous function such that, for every pair of simple closed curves $\lambda=(c,w)$ and $\lambda'=(c',w')$,
$$\iota(\lambda,\lambda')=w\cdot w'\cdot \#(\gamma\cap \gamma')~,$$
where $\gamma$ and $\gamma'$ are geodesic representatives of $c$ and $c'$ for any hyperbolic metric on $S$. 
\end{defi}

The following is the first step towards a relation between the volume of a maximal globally hyperbolic manifold and the length of the left and right earthquake laminations of the Earthquake Theorem (Theorem \ref{thm thurston}). 

\begin{lemma} \label{lemma volume and lams}
Given a maximal globally hyperbolic manifold $M=M_{h,h'}$, let $\lambda_l$ and $\lambda_r$ be the measured laminations such that $E_l^{\lambda_l}(h)=h'$ and $E_r^{\lambda_r}(h)=h'$. Then
\begin{equation} \label{eq complement 1}
\Vol(\mathcal C(M))+\Vol(M_+)=\frac{1}{4}\ell_{\lambda_r}(h)+\frac{\pi^2}{2}|\chi(S)|~,
\end{equation}
and 
\begin{equation} \label{eq complement 2}
\Vol(\mathcal C(M))+\Vol(M_-)=\frac{1}{4}\ell_{\lambda_l}(h)+\frac{\pi^2}{2}|\chi(S)|~.
\end{equation}
\end{lemma}

The proof follows from the arguments in \cite[Section 8.2.3]{bebo}.

\begin{repcor}{cor difference lengths}
Given two hyperbolic metrics $h$ and $h'$ on $S$, if $\lambda_l$ and $\lambda_r$ are the measured laminations such that $E_l^{\lambda_l}(h)=h'$ and $E_r^{\lambda_r}(h)=h'$, then
$$|\ell_{\lambda_l}(h)-\ell_{\lambda_r}(h)|\leq 2\pi^2|\chi(S)|~.$$
\end{repcor}
\begin{proof}
From Equations \eqref{eq complement 1} and \eqref{eq complement 2}, it follows that
$$\frac{1}{4}|\ell_{\lambda_l}(h)-\ell_{\lambda_r}(h)|=|\Vol(M_+)-\Vol(M_-)|\leq \max\{ \Vol(M_-),\Vol(M_+)\}\leq \frac{\pi^2}{2}|\chi(S)|~,$$
where the last inequality is the content of Proposition \ref{prop maximal volume}.
\end{proof}

\begin{reptheorem}{prop:comparison volume lams}
Given a maximal globally hyperbolic manifold $M_{h,h'}$, let $\lambda_l$ and $\lambda_r$ be the measured laminations such that $E_l^{\lambda_l}(h)=h'$ and $E_r^{\lambda_r}(h)=h'$. Then
\begin{equation} \label{eq volume convex core 1}
\frac{1}{4}\ell_{\lambda_l}(h)\leq 
\Vol(\mathcal C(M_{h,h'}))\leq \frac{1}{4}\ell_{\lambda_l}(h)+\frac{\pi^2}{2}|\chi(S)|~,
\end{equation}
and analogously
\begin{equation} \label{eq volume convex core 2}
\frac{1}{4}\ell_{\lambda_r}(h)\leq 
\Vol(\mathcal C(M_{h,h'}))\leq \frac{1}{4}\ell_{\lambda_r}(h)+\frac{\pi^2}{2}|\chi(S)|~.
\end{equation}
\end{reptheorem}
\begin{proof}
From Lemma \ref{lemma volume and lams}, we have
$$\Vol(\mathcal C(M))=\frac{1}{4}\ell_{\lambda_l}(h)+\frac{\pi^2}{2}|\chi(S)|-\Vol(M_-)~,$$
and thus the claim follows, using that
$$0\leq \Vol(M_-)\leq \frac{\pi^2}{2}|\chi(S)|$$
 by Proposition \ref{prop maximal volume}. The other inequality holds analogously.
\end{proof}



\section{Holomorphic energy}\label{sec:energiaL1}

In this section we will discuss the relation between the volume of a maximal globally hyperbolic Anti-de Sitter manifold and several types of 1-\emph{energy}, that is, the holomorphic 1-energy obtained by integrating the norm $||\partial f||$ of the $(1,0)$-part of the differential of a diffeomorphism $f$ between Riemannian surfaces, and the integral of the $1$-Schatten norm of the differential of $f$.

\subsection{Volume of a convex set bounded by $K$-surfaces}

As a consequence of Proposition \ref{prop maximal volume}, the volume of the convex core of a maximal globally hyperbolic $\AdS^{3}$ manifold $M_{h,h'}$ is coarsely equivalent to the volume of every domain of $M_{h,h'}$ in which it is contained. Using this fact, we will be able to compare the volume of $M_{h,h'}$ with the minima of certain functionals which depend on $(h,h')\in\Teich(S)\times\Teich(S)$.

As explained in Subsection \ref{subsec min lag maximal surf},
in a maximal globally hyperbolic $\AdS^{3}$ manifold, there exists a unique embedded maximal surface $\Sigma_{0}=\Sigma_{h,h'}$ (i.e with vanishing mean curvature) with principal curvatures in $(-1,1)$. 
By an application of the maximum principle, $\Sigma_{0}$ is contained in the convex core of $M_{h,h'}$. Moreover, using the formulas for the shape operator $B_{\rho}$ of equidistant surfaces (see \cite{Schlenker-Krasnov} or \cite[Lemma 1.14]{seppimaximal}), it is straightforward to verify that a foliation by equidistant surfaces $\Sigma_{\rho}$ from $\Sigma_{0}$ is defined at least for $\rho \in [-\frac{\pi}{4}, \frac{\pi}{4}]$ and the surfaces  $\Sigma_{-\frac{\pi}{4}}$ and $\Sigma_{\frac{\pi}{4}}$ are convex resp. concave, with constant Gaussian curvature $-2$. Therefore, the domain with boundary
\[	
	\Omega_{h,h'}=\bigcup_{\rho \in [-\frac{\pi}{4}, \frac{\pi}{4}]} \Sigma_{\rho}
\]
contains the convex core and by definition
\[
	\Vol(\Omega_{h,h'})=\int_{-\frac{\pi}{4}}^{\frac{\pi}{4}} \Area(\Sigma_{\rho}) d\rho \ .
\]
By exploiting the analytic relation between maximal surfaces in $\AdS^3$-manifolds and minimal Lagrangian diffeomorphisms between hyperbolic surfaces, we can express explicitly this volume as a functional of $h$ and $h'$. 



In fact (recalling Definition \ref{defi min lag}), the minimal Lagrangian map  $m:(S,h)\to(S,h')$ can be characterised in the following way, see \cite{labourieCP}:
\begin{lemma} \label{lemma char min lag} 
Given two hyperbolic metrics $h$ and $h'$ on $S$, an orientation-preserving diffeomorphism $m:(S,h)\to(S,h')$ is minimal Lagrangian if and only if there exists a bundle morphism $b\in\Gamma(\mathrm{End}(TS))$ such that
\begin{enumerate}
	\item $m^{*}h'=h(b\cdot, b\cdot)$ 
	\item $\det(b)=1$
	\item $b$ is $h$-self-adjoint
	\item $b$ satisfies the Codazzi equation $d^{\nabla}b=0$ for the Levi-Civita connection $\nabla$ of $h$.
\end{enumerate}
\end{lemma}
 Moreover, if we denote with $I_{0}$ the induced metric on $\Sigma_{0}$ and with $B_{0}$ its shape operator, and we identify $(\Sigma_0,I_0)$ and $(S,h)$ using the left projection, which is the unique harmonic map in the given isotopy class, the following relations hold (see \cite{Schlenker-Krasnov,bon_schl,BonSepKsurfacesAdS}):
\begin{equation}\label{eq:formulearea}
  I_{0}=\frac{1}{4}h((E+b)\cdot, (E+b)\cdot) \ \  \ \ \ \text{and} \ \ \ \  \ B_{0}=-(E+b)^{-1}J_{h}(E-b)~.
\end{equation}
Here $J_{h}$ is the complex structure on $S$ compatible with the metric $h$. In particular, it can be checked directly that the surface $\Sigma_{0}$ is maximal precisely when conditions $(1)-(4)$ of Lemma \ref{lemma char min lag} hold, that is, when the associated map is minimal Lagrangian. \\
\\
\indent Therefore, by using the above formulas and the fact that the metric on the parallel surface $\Sigma_{\rho}$ at distance $\rho$ from $\Sigma_{0}$ is given by
\[
	I_{\rho}=I_0((\cos(\rho)E+\sin(\rho)B_{0})\cdot, (\cos(\rho)E+\sin(\rho)B_{0})\cdot) \ ,
\]
the area form of $(\Sigma_\rho,I_\rho)$ is 
$$
	d\mathrm{A}_{\Sigma_\rho}=\det(\cos(\rho)E+\sin(\rho)B_{0})d\mathrm{A}_{\Sigma_0}=(\cos^{2}(\rho)+\sin^{2}(\rho)(\det B_0))d\mathrm{A}_{\Sigma_0}~.
$$
Moreover, from Equation \eqref{eq:formulearea}, we have
\[
	d\mathrm{A}_{\Sigma_0}=\frac{1}{4}\det(E+b)d\mathrm{A}_{h} \ \  \ \ \ \text{and} \ \ \ \  \ \det B_0=\frac{\det(E-b)}{\det(E+b)}=\frac{2-\trace b}{2+\trace b}~.
\]
Therefore we get: 
\begin{align*}
	\Area(\Sigma_{\rho})&=\int_{\Sigma_{0}}(\cos^{2}(\rho)+\sin^{2}(\rho)(\det B_0))d\mathrm{A}_{\Sigma_0}\\
		&=\cos^{2}(\rho)\int_{\Sigma_{0}}d\mathrm{A}_{\Sigma_0}+\sin^{2}(\rho)\int_{\Sigma_{0}}\det(B_{0})d\mathrm{A}_{\Sigma_0}\\
		&=\frac{\cos^{2}(\rho)}{4}\int_{\Sigma_{0}}\det(E+b)d\mathrm{A}_{h}+\frac{\sin^{2}(\rho)}{4}\int_{\Sigma_{0}}\det(E-b)d\mathrm{A}_{h}\\
		&=\frac{\cos^{2}(\rho)}{4}\int_{\Sigma_{0}}(2+\trace(b))d\mathrm{A}_{h}+\frac{\sin^{2}(\rho)}{4}\int_{\Sigma_{0}}(2-\trace(b))d\mathrm{A}_{h}\\
		&=\pi|\chi(S)|+\frac{1}{4}(\cos^{2}(\rho)-\sin^{2}(\rho))\int_{\Sigma_{0}}\trace(b)d\mathrm{A}_{h} \ ,
\end{align*}
where in the last step we used the Gauss-Bonnet equation for the hyperbolic metric $h$. 
Integrating for $\rho \in [-\pi/4, \pi/4]$ we have
\begin{equation} \label{eq filo sganga}
	\Vol(\Omega_{h,h'})=\frac{\pi^{2}}{2}|\chi(S)|+\frac{1}{4}\int_{\Sigma_{0}}\trace(b)d\mathrm{A}_{h} \ .
\end{equation}
Recall that from Theorem \ref{thm min lag}, there exists a unique minimal Lagrangian map $m:(S,h)\to(S,h')$ isotopic to the identity between any two hyperbolic surfaces $(S,h)$ and $(S,h')$. Hence we can now prove:

\begin{cor}\label{cor:confrontotracciab}
Let $M_{h,h'}$ be a maximal globally hyperbolic  $\AdS^{3}$ manifold. Let $b:TS\rightarrow TS$ be the unique $h$-self-adjoint Codazzi operator such that $m^*h'=h(b\cdot, b\cdot)$, where $m:(S,h)\to(S,h')$ is the minimal Lagrangian diffeomorphism. Then
\[
	\frac{1}{4}\int_{S}\trace(b)d\mathrm{A}_{h}-{\pi}|\chi(S)| \leq \Vol(\mathcal{C}(M_{h,h'}))\leq \frac{1}{4}\int_{S}\trace(b)d\mathrm{A}_{h}+\frac{\pi^{2}}{2}|\chi(S)| ~.
\]
\end{cor}
\begin{proof} By the previous computation, we have
\[
	\Vol(\mathcal{C}(M_{h,h'}))\leq \Vol(\Omega_{h,h'})=\frac{\pi^{2}}{2}|\chi(S)|+\frac{1}{4}\int_{S}\trace(b)d\mathrm{A}_{h} \ .
\]
On the other hand, by an adaptation of the proof of Proposition \ref{prop maximal volume}, since the boundary of $\Omega_{h,h'}$ consists of the disjoint union of the two surfaces with constant curvature $-2$ in $M_{h,h'}$, for every $\epsilon>0$, we have
\begin{align*}
\Vol(\Omega_{h,h'}\setminus \mathcal{C}(M_{h,h'}))&\leq 2{(1+\epsilon)}\int_0^{1} \frac{1}{2\sqrt\kappa(\kappa+1)}\int_{S_\kappa}{d\mathrm{Area}_{S_\kappa}}d\kappa \\
&=2\pi|\chi(S)|{(1+\epsilon)}\int_0^{1} \frac{d\kappa}{\sqrt\kappa(\kappa+1)^2}=|\chi(S)|\left(\pi+\frac{\pi^2}{2}\right){(1+\epsilon)}~,
\end{align*}
and therefore
$$\Vol(\Omega_{h,h'}\setminus \mathcal{C}(M_{h,h'}))\leq |\chi(S)|\left(\pi+\frac{\pi^2}{2}\right)~.$$
Hence, using Equation \eqref{eq filo sganga},
\begin{align*}
	\Vol(\mathcal{C}(M_{h,h'}))&=\Vol(\Omega_{h,h'})-\Vol(\Omega_{h,h'}\setminus\mathcal{C}(M_{h,h'})) \\
			&\geq \Vol(\Omega_{h,h'})-|\chi(S)|\left(\pi+\frac{\pi^2}{2}\right)= \frac{1}{4}\int_{S}\trace(b)d\mathrm{A}_{h}-{\pi}|\chi(S)| \ ,
\end{align*}
as claimed.

\end{proof}

\subsection{Holomorphic energy and Schatten energy} \label{subsec hol 1-energy}
As a consequence of Corollary \ref{cor:confrontotracciab}, the coarse properties of the volume of a maximal globally hyperbolic $\AdS^{3}$ manifold depend only on the function
\begin{align*}
	F: \Teich(S)\times \Teich(S) &\rightarrow \R^{+}\\
		(h,h') &\mapsto \int_{S}\trace(b) d\mathrm{A}_{h}~,
\end{align*}
where $b$ is the Codazzi tensor, satisfying the conditions $(1)-(4)$ of Lemma \ref{lemma char min lag} above, for the minimal Lagrangian map $m:(S,h)\to(S,h')$. The properties of $F$ have already been introduced and studied in \cite{bms2}. Here we point out the relation with an $L^{1}$-energy on Teichm\"uller space.\\
\\
Let us denote by $C^1_\mathrm{id}(S)$ the space of $C^1$ maps $f:(S,h)\rightarrow (S,h')$ homotopic to the identity. Equivalently, by identifying $(S,h)$ with $\Hyp^2/\rho(\pi_1(S))$ and $(S,h')$ with $\Hyp^2/\rho'(\pi_1(S))$ (where $\rho,\rho'$ are the holonomy representations of $\pi_1(S)$ into $\isom(\Hyp^2)$), $C^1_\mathrm{id}(S)$ coincides with the space of $(\rho,\rho')$-equivariant $C^1$ maps of $\Hyp^{2}$ into itself.

Let us recall that the $L^1$-energy, or total variation, is defined as:

\begin{defi} \label{defi energial1} Given two hyperbolic surfaces $(S,h)$ and $(S,h')$, the $L^1$-\emph{energy}, or \emph{total variation}, of $f$ is the functional $$E_{d}(\cdot,h,h'):C^1_{\mathrm{id}}(S) \rightarrow \R^{+}$$
defined by
\[
	E_{d}(f,h,h')=\int_{S}||df|| d\mathrm{A}_{h}~,
\]
where, if  $df^{*}$ is the $h$-adjoint operator of $df$, then
$$||df||=\sqrt{\tr(df^*df)}~.$$
\end{defi}

\begin{remark}
We remark that, with the above definition, if $f:(S,h)\to(S,h')$ is an isometry, then at every point $x\in (S,h)$, $||df||=\sqrt 2$. Therefore, the energy of an isometry $f$ between two hyperbolic surfaces is $2\sqrt 2\pi|\chi(S)|$. 
\end{remark}

In this paper, we will consider also another $L^1$-type functional, which is the following $1$-Schatten energy:

\begin{defi} Given two hyperbolic surfaces $(S,h)$ and $(S,h')$, the $1$-\emph{Schatten energy} is the functional $E_{Sch}(\cdot,h,h'):C^1_\mathrm{id}(S) \rightarrow \R^{+}$
\[
	E_{Sch}(f,h,h')=\int_{S}\trace(b_{f}) d\mathrm{A}_{h}
\]
where,  if  $df^{*}$ is the $h$-adjoint operator of $df$, then $b_{f}$ is the square root of $df^{*}df$.  
\end{defi}

\begin{remark} \label{rmk trace and schatten norm}
The quantity $\trace(b_f)$ coincides with the $1$-Schatten norm of the operator $df$ at $x$, and this justifies the definition of $E_{Sch}(f)$ as the $1$-Schatten energy.
At every point $x\in S$, the tensor $b_f$ at $x$ is the unique $h$-self-adjoint operator such that $f^{*}h'=h(b_{f}\cdot, b_{f}\cdot)$. 
\end{remark}

\begin{remark} \label{rmk schatten} The $1$-Schatten energy of a $C^1$ map $f$ is related to the \emph{holomorphic energy} 
\[
	E_{\partial}(f,h,h')=\int_{S}\|\partial f\| d\mathrm{A}_{h}
\]
studied by Trapani and Valli in \cite{oneharmonic}. (Here we are considering $\partial f$ as a holomorphic $1$-form with values in $f^{*}TS$ and we denote with $\| \cdot \|$ the norm on $T^{*}S\otimes f^{*}TS$ induced by the metrics $h$ and $h'$, as in Definition \ref{defi energial1}).
A computation in local coordinates, using Remark {\ref{rmk trace and schatten norm}}, shows that the eigenvalues of $df^{*}df$ are
\[
    \mu_{1}=\frac{1}{2}(\| \partial f\|-\|\overline{\partial}f\|)^{2}  \ \ \ \text{and} \ \ \ \mu_{2}=\frac{1}{2}(\|\partial f\|+\|\overline{\partial}f\|)^{2} \ .
\]
Therefore one obtains:
\begin{equation} \label{eq trace and maximum}
\trace(b_{f})={\sqrt 2}\max\{||\partial f||,||\overline\partial f||\}~.
\end{equation}
In particular, when $f$ is orientation-preserving (for instance if $f$ is a minimal Lagrangian diffeomorphism), then $\|\partial f\|^{2}-\|\overline{\partial}f\|^{2}>0$ and therefore
$$\trace(b_{f})={\sqrt 2}\|\partial f\| \ .$$
In conclusion, this shows that
\begin{equation} \label{eq comparison schatten 1-hol}
E_\partial(f,h,h')\leq\frac{\sqrt 2}{2}E_{Sch}(f,h,h')~,
\end{equation}
with equality when $f$ is an orientation-preserving diffeomorphism.
\end{remark}
Let us denote by $\Diffeo_{0}(S,h,h')$ the space of orientation preserving diffeomorphisms $f:(S,h)\rightarrow (S,h')$ isotopic to the identity. Trapani and Valli proved that the holomorphic 1-energy $E_{\partial}(\cdot,h,h')$ is minimized on $\Diffeo_{0}(S,h,h')$ by the unique minimal Lagrangian map $m:(S,h)\to (S,h')$:

\begin{prop}\cite[Lemma 3.3]{oneharmonic}\label{prop:minlag} Given two hyperbolic metrics $(S,h)$ and $(S,h')$, the functional
$$E_{\partial}(\cdot,h,h'):\Diffeo_{\mathrm{id}}(S,h,h') \rightarrow \R^{+}$$
admits a unique minimum attained by the minimal Lagrangian map $m:(S,h)\to(S,h')$ isotopic to the identity.
\end{prop}

We will actually need the fact that the minimal Lagrangian map $m:(S,h)\to(S,h')$ also minimizes $E_{Sch}$ on $C^1_\mathrm{id}(S)$, which is an improvement of Proposition \ref{prop:minlag}:

\begin{prop}\label{prop:minlag sch} Given two hyperbolic metrics $(S,h)$ and $(S,h')$, the functional
$$E_{Sch}(\cdot,h,h'):C^1_\mathrm{id}(S) \rightarrow \R^{+}$$
admits a minimum attained by the minimal Lagrangian map $m:(S,h)\to(S,h')$ isotopic to the identity.
\end{prop}
The proof follows from the convexity of the functional $E_{Sch}$, see \cite{bonschlpreparation}. In fact, 
the space $\Diffeo_{\mathrm{id}}(S,h,h')$ of diffeomorphisms isotopic to the identity is open in $C^\infty_\mathrm{id}(S)$ (i.e. the space of $C^\infty$ self maps of $S$ homotopic to the identity). Moreover, by Remark \ref{rmk schatten}, $E_{Sch}$ and $E_\partial$ coincide on $\Diffeo_{\mathrm{id}}(S,h,h')$, up to a factor. By Proposition \ref{prop:minlag}, $m$ is a local minimum of $E_{Sch}$ on $C^\infty_\mathrm{id}(S)$, and thus a global minimum on $C^{\infty}_{\mathrm{id}}(S)$ by convexity. By density of $C^\infty_\mathrm{id}(S)$ in $C^1_\mathrm{id}(S)$, and the continuity of $E_{Sch}$ on $C^1_\mathrm{id}(S)$, it follows that $m$ is a global minimum of $E_{Sch}(\cdot,h,h')$ on $C^1_\mathrm{id}(S)$, as well.

The above results enable us to conclude the following theorem:

\begin{reptheorem}{theorem hol energy}
Let $M_{h,h'}$ be a maximal globally hyperbolic $\AdS^3$ manifold. Then
$$\frac{1}{4}E_{Sch}(m,h,h')-\pi|\chi(S)| \leq \Vol(\mathcal{C}(M_{h,h'}))\leq \frac{1}{4}E_{Sch}(m,h,h')+\frac{\pi^{2}}{2}|\chi(S)|~,$$
where $m:(S,h)\to (S,h')$ is the minimal Lagrangian map isotopic to the identity, that is, the minimum of the 1-Schatten energy functional $E_{Sch}(\cdot,h,h'):C^1_{\mathrm{id}}(S)\to\R$. 
\end{reptheorem}
\begin{proof}
By Proposition \ref{prop:minlag sch}, the minimum of $E_{Sch}(\cdot,h,h')$ is achieved at the minimal Lagrangian map $m$, and we have by definition
$$E_{Sch}(m,h,h')=\int_S \trace(b)d\mathrm A_h~.$$
Hence the statement follows from Corollary \ref{cor:confrontotracciab}.
\end{proof}

\section{$L^1$-energy between hyperbolic surfaces} \label{sec main thm}
We conclude this section by showing that the volume of maximal globally hyperbolic $\AdS^3$ manifolds is also coarsely comparable to the $L^1$-energy on Teichm\"uller space, which is the content of Theorem \ref{cor energiaL1}. Recall that the definition of $L^1$-energy functional was given in Definition \ref{defi energial1}.

For the proof of the inequality 
$$\frac{1}{4}\inf_{f\in C^1_\mathrm{id}(S)}E_d(f,h,h') -\frac{\sqrt{2}}{2}\pi|\chi(S)|\leq \Vol(\mathcal{C}(M_{h,h'}))$$
of Theorem \ref{cor energiaL1}, we will need the fact that
$$\inf_{f\in C^1_\mathrm{id}(S)}E_d(f,h,h')\leq \ell_{\lambda}(h)+2\sqrt{2}\pi|\chi(S)|~,$$
so as to apply Theorem \ref{prop:comparison volume lams}. This follows from the existence of a sequence $f_n\in C^1_{\mathrm{id}}(S)$ (actually the sequence is in $C^\infty_{\mathrm{id}}(S)$), converging to the earthquake map $e_\lambda:(S,h)\to(S,h')$, such that 
$\lim E_d(f_n,h,h')\leq \ell_{\lambda}(h)+2\sqrt{2}\pi|\chi(S)|$.

\begin{lemma} \label{lemma variation earthquake}
Given two hyperbolic surfaces $(S,h)$ and $(S,h')$, let $\lambda$ be the measured  lamination such that $E_l^{\lambda}(h)=h'$ (or $E_r^{\lambda}(h)=h'$). Then there exists a sequence $f_n\in C^1_{\mathrm{id}}(S)$ such that
$$\lim_{n\to +\infty} E_d(f_n,h,h')\leq \ell_{\lambda}(h)+2\sqrt{2}\pi|\chi(S)|~.$$
\end{lemma}
\begin{proof}
We will give the proof for left earthquakes. 
Suppose first that $\lambda$ is a weighted simple closed geodesic $(\gamma,w)$. Let $U_\epsilon$ be the $\epsilon$-neighborhood of $\gamma$ on $(S,h)$. Choose coordinates $(t,r)$ on $U_\epsilon$, so that the geodesic $\gamma$ is parameterized by arclength by the coordinate $(t,0)$, {for $t \in [0,L]$}, and the point $(t,r)$ is at signed distance $r$ from the point $(t,0)$. Hence the metric on $U_\epsilon$ has the form $dr^2+\cosh^2 (r)dt^2$.

 Then define $f_\epsilon(r,t)=(r,t+g_\epsilon(r))$ on $U_\epsilon$, where $g_\epsilon(r)$ is a smooth increasing map such that $g_\epsilon(-\epsilon)=0$ and $g_\epsilon(\epsilon)=w$. By definition of earthquake map, we can then extend $f_\epsilon$ to be an isometry on $S\setminus U_\epsilon$.  By a direct computation, 
$$||df_{\epsilon(r,t)}||=\sqrt{2+g_\epsilon'(r)^2}~,$$
hence 
\begin{align*}
\int_{U_\epsilon} ||df_\epsilon||d\mathrm A_h&=L\int_{-\epsilon}^{\epsilon} \sqrt{2+g_\epsilon'(r)^2}\cosh (r)dr\leq L\int_{-\epsilon}^{\epsilon}(\sqrt 2+g_\epsilon'(r))\cosh (r) dr \\
&\leq \sqrt 2 \mathrm{Area}(U_\epsilon)+L\cosh(\epsilon)(g(\epsilon)-g(-\epsilon))~.
\end{align*}
Therefore, using that $L(g(\epsilon)-g(-\epsilon))=Lw=\ell_\lambda(h)$ and that $f_\epsilon$ is an isometry outside of $U_\epsilon$, we get
\begin{align*}
\int_S||df_\epsilon||d\mathrm A_h&=\int_{S\setminus U_\epsilon}||df_\epsilon||d\mathrm A_h+\int_{U_\epsilon}||df_\epsilon||d\mathrm A_h \\
&\leq \cosh(\epsilon)\ell_\lambda(h)+\sqrt 2\mathrm{Area}(S\setminus U_\epsilon)+\sqrt 2\mathrm{Area}(U_\epsilon) \\
&=\cosh(\epsilon)\ell_\lambda(h)+\sqrt 2\mathrm{Area}(S)=\cosh(\epsilon)\ell_{\lambda}(h)+2\sqrt{2}\pi|\chi(S)|~.
\end{align*}
As we let $\epsilon\to 0$, this concludes that
$$\lim_{\epsilon\to 0}\int_S||df_\epsilon||d\mathrm A_h\leq \ell_{\lambda}(h)+2\sqrt{2}\pi|\chi(S)|~.$$
Taking a sequence with $\epsilon=1/n$, this concludes the proof of the statement in the case $\lambda$ is a weighted simple closed curve.

Let us now take an arbitrary measured geodesic lamination $\lambda$. Let $\lambda_n$ be a sequence of weighted multicurves converging to $\lambda$, so that:
\begin{itemize}
\item $|\ell_{\lambda_n}(h)-\ell_\lambda(h)|\leq 1/n$.
\item The metrics $h_n=E_l^{\lambda_n}(h)$ and $h'=E_l^{\lambda}(h)$ are $(1+1/n)$-bi-Lipschitz.
\end{itemize}
In fact, the second step follows from the continuity of the earthquake map $E_l:\mathcal {ML}(S)\times\Teich(S)\to\Teich(S)$. Let us now take $f_n:(S,h)\to(S,h_n)$ (constructed as before) so that 
$$\int_S ||df_n||d\mathrm A_h\leq \ell_{\lambda_n}(h)+2\sqrt{2}\pi|\chi(S)|+\frac{1}{n}~.$$
Let $g_n:(S,h_n)\to(S,h')$ be the $(1+1/n)$-bi-Lipschitz diffeomorphisms. Since $h_n\to h'$, we can assume $g_n\to\mathrm{id}$. Then for the map $g_n\circ f_n:(S,h)\to(S,h')$, we have:
\begin{align*}
\int_S ||d(g_n\circ f_n)||d\mathrm A_h&\leq\left(1+\frac{1}{n}\right) \int_S ||d f_n||d\mathrm A_h \\
&\leq \left(1+\frac{1}{n}\right)\left(\ell_{\lambda_n}(h)+2\sqrt{2}\pi|\chi(S)|+\frac{1}{n}\right) \\
&\leq \left(1+\frac{1}{n}\right)\left(\ell_{\lambda}(h)+2\sqrt{2}\pi|\chi(S)|+\frac{2}{n}\right)~.
\end{align*}
Hence the constructed sequence $g_n\circ f_n:(S,h)\to(S,h')$ converges to the earthquake map $e_\lambda$  and satisfies: 
$$\lim_{n\to+\infty} \int_S ||d(g_n\circ f_n)||d\mathrm A_h\leq  \ell_{\lambda}(h)+2\sqrt{2}\pi|\chi(S)|~,$$
hence concluding the proof.
\end{proof}

\begin{remark}
Although not strictly necessary in this paper, we remark that $E_d$ can be extended to a lower-semicontinuous functional on the space of $L^{2}$ maps from $S$ to itself
$$V(\cdot,h,h'):L^2(S,S)\to\R\cup\{+\infty\}~,$$
which is defined by
$$V(f,h,h')=\liminf_{g\to f} E_d(g,h,h')~.$$
The functional $V(\cdot,h,h')$ is called \emph{total variation}, and the maps where it takes finite values are \emph{bounded variation maps}. 
Lemma \ref{lemma variation earthquake} shows the inequality $V(e_{\lambda},h,h')\leq \ell_{\lambda}(h)+2\sqrt{2}\pi|\chi(S)|$. 
\end{remark}

We are now able to prove the main result connecting the volume of maximal globally hyperbolic $\AdS^3$ manifolds with the minima of the $L^1$-energy:  
\begin{reptheorem}{cor energiaL1}
Let $M_{h,h'}$ be a maximal globally hyperbolic  $\AdS^{3}$ manifold. Then
$$\frac{1}{4}\inf_{f\in C^1_\mathrm{id}(S)}E_d(\cdot,h,h') -\frac{\sqrt{2}}{2}\pi|\chi(S)|\leq \Vol(\mathcal{C}(M_{h,h'}))\leq \frac{\sqrt{2}}{2}\inf_{f\in C^1_\mathrm{id}(S)} E_d(\cdot,h,h')+\frac{\pi^{2}}{2}|\chi(S)|~.$$
\end{reptheorem}
\begin{proof}
From Lemma \ref{lemma variation earthquake} and Theorem \ref{prop:comparison volume lams}, we have $$\frac{1}{4}\inf_{f\in C^1_{\mathrm{id}}(S)}E_d(\cdot,h,h')\leq \frac{1}{4}(\ell_\lambda(h)+2\sqrt{2}\pi|\chi(S)|)\leq \Vol(\mathcal{C}(M_{h,h'}))+\frac{\sqrt{2}}{2}\pi|\chi(S)|~,$$
hence the lower bound follows. On the other hand, using the fact that 
$$||df||^2=||\partial f||^2+||\overline \partial f||^2~,$$
 from Equation \eqref{eq trace and maximum} 
we have for every $f\in C^1_{\mathrm{id}}(S)$: 
$$\trace(b_{f})={\sqrt 2}\max\{||\partial f||,||\overline\partial f||\}\leq \sqrt 2 ||df||~.$$
Thus
$$E_{Sch}(f,h,h')=\int_S \trace(b_{f})dA_h\leq \sqrt 2\int_S ||df||dA_h~.$$
Hence the upper bound follows from Theorem \ref{theorem hol energy} and Proposition \ref{prop:minlag sch}:
\begin{align*}
\Vol(\mathcal{C}(M_{h,h'})) &\leq \frac{\pi^{2}}{2}|\chi(S)|+\frac{1}{4}E_{Sch}(m,h,h') \\
&=\frac{\pi^{2}}{2}|\chi(S)|+\frac{1}{4}\inf_{f\in C^1_\mathrm{id}(S)} E_{Sch}(f,h,h') \\ 
&\leq \frac{\pi^{2}}{2}|\chi(S)|+\frac{\sqrt 2}{4}\inf_{f\in C^1_\mathrm{id}(S)} E_d(f,h,h')~,
\end{align*}
thus concluding the proof.
\end{proof}

\section{Thurston's asymmetric distance}\label{sec:distanzaThurston}

In this Section, we will apply Corollary \ref{cor:confrontotracciab} to compare the volume of the convex core of a maximal globally hyperbolic $\AdS^{3}$-manifold and Thurston's asymmetric distance on Teichm\"uller space.

\subsection{The general upper bound}
 Thurston asymmetric distance on Teichm\"uller space is deeply related to the hyperbolic geometry of surfaces. We briefly recall here the main definitions for the convenience of the reader.\\
\\
\indent Let $h$ and $h'$ two hyperbolic metrics on $S$. Given a diffeomorphism isotopic to the identity $f:(S,h) \rightarrow (S,h')$ we define the Lipschitz constant of $f$ as
\[
	L(f)=\sup_{x \neq y \in S}\frac{d_{h'}(f(x), f(y))}{d_{h}(x,y)} \ .
\]

\begin{defi}Thurston asymmetric distance between $h,h'\in \Teich(S)$ is 
\[
	\dth(h,h')=\inf_{f\in\mathrm{Diff}_{\mathrm{id}}}\log(L(f))
\]
where the infimum is taken over all diffeomorphisms $f:(S,h)\rightarrow (S,h')$ isotopic to the identity.
\end{defi}

Thurston showed that the Lipschitz constant $L(f)$ can also be computed by comparing lengths of closed geodesics for the metrics $h$ and $h'$. More precisely, in \cite{Thurstondistance} he proved that
\begin{equation} \label{eq thuston lipschitz}
	L(f)=\sup_{c}\frac{\ell_{c}(h')}{\ell_{c}(h)}  \ ,
\end{equation}
where $c$ varies over all simple closed curves $c$ in $S$.

An application of Theorem \ref{theorem hol energy} leads to the following comparison between the volume of the convex core of a maximal globally hyperbolic $\AdS^{3}$ manifold and Thurston asymmetric distance.

\begin{reptheorem}{cor:confrontovolumeThurston} Let $M_{h,h'}$ be a maximal globally hyperbolic $\AdS^{3}$ manifold. Let $M_{h,h'}$ be a maximal globally hyperbolic $\AdS^{3}$ manifold. Then
\[
	\Vol(\mathcal{C}(M_{h,h'})) \leq \frac{\pi^{2}}{2}|\chi(S)|+\pi|\chi(S)|\exp(\min\{\dth(h,h'),\dth(h',h)\}) \ .
\]
\end{reptheorem}
\begin{proof} 
We will first prove that 
$$\Vol(\mathcal{C}(M_{h,h'})) \leq \frac{\pi^{2}}{2}|\chi(S)|+\pi|\chi(S)|\exp(\dth(h,h'))~.$$
First of all, let us observe that the Lipschitz constant of a diffeomorphism $f:(S,h)\to(S,h')$ can be expressed as:
$$L(f)=\sup_{v\in TS}\frac{||df(v)||_{h'}}{||v||_h}=\sup_{x\in S}||df_x||_{\infty}~,$$
Here, $||df_x||_{\infty}$ is the spectral norm of $df_x:(T_x S,h_x)\to(T_{f(x)}S,h'_{f(x)})$. Now, from Theorem \ref{theorem hol energy}, for every diffeomorphism $f:(S,h)\rightarrow (S,h')$ isotopic to the identity, we have
$$\Vol(\mathcal{C}(M_{h,h'}))\leq  \frac{\pi^{2}}{2}|\chi(S)|+\frac{1}{4}E_{Sch}(f,h,h')~.$$
Since the spectral norm of $df$ is the maximum eigenvalue of $\sqrt{df^{*}df}$, the $1$-Schatten norm is bounded by twice the spectral norm, hence we get
$$E_{Sch}(f,h,h')=\int_S\trace(b_f)d\mathrm{A}_h\leq 2\sup_{x\in S}||df_x||_{\infty}\int_S d\mathrm{A}_h=4\pi|\chi(S)|L(f)~.$$
Hence we obtain:
\[
	\Vol(\mathcal{C}(M_{h,h'}))\leq \frac{\pi^{2}}{2}|\chi(S)|+\pi|\chi(S)|\inf_f L(f)=\frac{\pi^{2}}{2}|\chi(S)|+{\pi}|\chi(S)|e^{\dth(h,h')}~.
\]
For the main statement, observe that the involution
\begin{align*}
	\SL(2,\R) &\rightarrow \SL(2,\R) \\
	     A &\mapsto A^{-1}
\end{align*}
induces an orientation-reversing isometry of $\AdS^{3}$ which swaps the left and right metric in Mess' parameterization (see Section \ref{sec:volume}). Therefore, the volumes of the convex cores of $M_{h,h'}$ and $M_{h',h}$ are equal. Hence it follows that 
$$\Vol(\mathcal{C}(M_{h,h'})) \leq \frac{\pi^{2}}{2}|\chi(S)|+\pi|\chi(S)|\exp(\dth(h',h))$$
is also true. This concludes the proof.
\end{proof}

\subsection{A negative result}

We are now showing that it is not possible to find a lower-bound for the volume of the convex core in terms of the Thurston asymmetric distance between the left and right metric.

\begin{prop}\label{prop:controesempioThurston} There is no continuous, proper function $g:\R^{+}\rightarrow \R^{+}$ such that
\[
	g(\min\{\dth(h,h'), \dth(h',h)\})\leq\Vol(\mathcal{C}(M_{h,h'})) \ .
\]
for every couple of metrics $h, h' \in \Teich(S)$.
\end{prop}
\begin{proof}It is sufficient to show that it is possible to find a sequence of maximal globally hyperbolic $\AdS^{3}$ manifolds such that the volume of the convex core remains bounded but both Thurston's asymmetric distances between the left and right metric diverge. 

Choose a simple closed curve $\mu\in P$ that disconnects the surface in such a way that one connected component $S_{1}$ is a surface of genus $1$ with geodesic boundary equal to $\mu$. Fix a pants decomposition $P$ containing the curve $\mu$. Let $\alpha \subset S_{1}$ be the curve in the pants decomposition of $S$ contained in the interior of $S_{1}$. Fix a simple closed curve $\beta$ in $S_{1}$ (see Figure \ref{fig:controesempioThurston}) which intersects $\alpha$ in exactly one point. Choose then a hyperbolic metric $h$ on $S$ such that the geodesic representative of $\beta$  intersects $\alpha$ orthogonally. For every $n \in \mathbb{N}$ we define an element $h_{n} \in \Teich(S)$ with the property that all Fenchel-Nielsen coordinates of $h_{n}$ coincide with those of $h$ but the length of the curve $\alpha$, which we impose to be equal to ${1}/{n}$. In particular, the $h_n$-geodesic representative of $\beta$ intersects $\alpha$ orthogonally for every $n$.

\begin{figure}[h!]
\begin{tikzpicture}[scale=0.7]

\draw [blue] (3,-2.5) ellipse (0.5 and 1.5);
\draw [black] plot[smooth, tension=.7] coordinates {(3,-1) (0.5,-1) (-6.5,-0.5) (-10,-2.5) (-6.5,-4.5) (0.5,-4) (3,-4)};
\draw [black] plot[smooth, tension=.7] coordinates {(-6,-2.5) (-5,-2) (-4,-2.5)};
\draw [black] plot[smooth, tension=.7] coordinates {(-5.6,-2.2) (-5,-2.5) (-4.4,-2.2)};
\draw [green] plot[smooth, tension=.7] coordinates {(-4.8,-2.5) (-4.2,-3.2) (-4.2,-4.2) (-4.3,-4.5)};
\draw [dashed,green] plot[smooth, tension=.7] coordinates {(-4.3,-4.5) (-4.8,-3.9) (-5,-3.4) (-4.9,-2.5) };
\draw [red] plot[smooth, tension=.7] coordinates {(-4,-3.5) (-6.5,-3.5) (-8,-2.5) (-6,-1.5) (-2.5,-2) (-2,-3) (-4,-3.5)};
\node at (3.7,-2.0) [blue] {$\mu$};
\node at (-4, -3.2) [green] {$\alpha$};
\node at (-7.7,-2.5) [red] {$\beta$};
\end{tikzpicture}
\caption{\small{Curves described in the proof of Proposition \ref{prop:controesempioThurston}}}
\label{fig:controesempioThurston}
\end{figure}
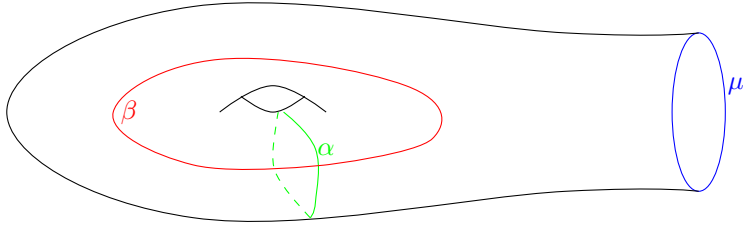

Consider the measured geodesic laminations $\lambda_{n}$ consisting of the simple closed curve $\alpha$ with weight $n$. We define a second sequence of hyperbolic metrics $h'_{n}$ as $h_{n}'=E_{l}^{\lambda_{n}}(h_{n})$. Notice that these metrics are obtained from $h_{n}$ by performing $n^{2}$ Dehn-twists along $\alpha$. We are going to show that the volume of the convex core of the maximal globally hyperbolic $\AdS^{3}$ manifolds $M_{n}=M_{h_{n},h_{n}'}$ remains bounded but the two Thurston's asymmetric distances between $h_{n}$ and $h_{n}'$ go to infinity when $n$ tends to $+\infty$. \\
By Equation \eqref{eq volume convex core 1} in Theorem \ref{prop:comparison volume lams}, the volume of the convex core of $M_{n}$ is coarsely equivalent to the length of $\lambda_{n}$, which by definition is
\[
	\ell_{\lambda}(h_{n})=\ell_{\alpha}(h_{n})\cdot\frac{1}{n}=1 \ ,
\]
hence the volume remains bounded.

On the other hand, since the curve $\beta$ intersects $\alpha$ orthogonally, for every metric $h_{n}$ we claim that 
\[
	\ell_{\beta}(h_{n})=4\arcsinh\left(\frac{\cosh(\frac{\ell_{\mu}(h)}{4})}{\sinh(\frac{1}{2n})}\right)  \ .
\]
To prove the claim, we cut the surface $S_{1}$ along the curve $\alpha$, thus obtaining a pair of pants $P'$ with geodesic boundaries given by $\mu$ and two copies of $\alpha$. If we cut again $P'$ into two right-angled hexagons (see Figure \ref{fig:controesempioThurston2}), the length of the curve $\beta$ can be computed using standard hyperbolic trigonometry \cite{thurston2}. Here we are also using the fact that the length of the curve $\mu$ does not depend on $n$.

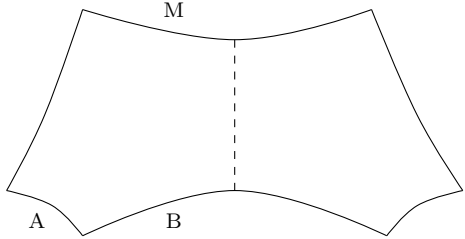
\begin{figure}[h!]
\begin{tikzpicture}

\draw  plot[smooth, tension=.7] coordinates {(-2,2.4) (0,2) (1.8,2.4)};
\draw  plot[smooth, tension=.7] coordinates {(-2,2.4) (-2.5,1) (-3,0)};
\node at (-0.8,2.4) {\footnotesize M};
\node at (-2.6,-0.4) {\footnotesize A};
\node at (-0.8, -0.4){\footnotesize B};
\draw  plot[smooth, tension=.7] coordinates {(-3,0) (-2.4,-0.2) (-2,-0.6)};
\draw  plot[smooth, tension=.7] coordinates {(1.8,2.4) (2.4,1) (3,0)};
\draw  plot[smooth, tension=.7] coordinates {(3,0) (2.4,-0.2) (2,-0.6)};
\draw  plot[smooth, tension=.7] coordinates {(-2,-0.6) (0,0) (2,-0.6)};
\draw [dashed] plot[smooth, tension=.7] coordinates {(0,2) (0,0)};

\end{tikzpicture}

\caption{\small{
The lengths of the edges $A$,$B$ and $M$ satisfy $\sinh(A)\sinh(B/2)=\cosh(M/2)$.}}
\label{fig:controesempioThurston2}

\end{figure}

\noindent Hence we obtain
\begin{equation}\label{eq:stima asintotica curva}
	\ell_{\beta}(h_{n})\leq C_1\left|\log(n)\right| +C_2 \ ,
\end{equation}
for some constants $C_1,C_2$, when $n$ is sufficiently big. 
Moreover, by a simple application of the triangle inequality (see \cite[Lemma 7.1]{BonSchlGAFA2009}, and recall Definition \ref{defi intersection} for the intersection of measured geodesic laminations), we can deduce that
\[
	\ell_{\beta}(h_{n}')+\ell_{\beta}(h_{n})\geq \iota(\lambda_{n}, \beta)=n \ ,
\]
thus
\[
	\frac{\ell_{\beta}(h_{n}')}{\ell_{\beta}(h_{n})}\geq -1+\frac{n}{\ell_{\beta}(h_{n})} \to +\infty
\]
when $n$ tends to $+\infty$ by using Equation \eqref{eq:stima asintotica curva}. Therefore, by definition $\dth(h_{n}, h_{n}') \to +\infty$ \ .\\
To prove that also $\dth(h_{n}', h_{n})$ is unbounded, it is sufficient to repeat the same argument for the curve $\beta'=D^{n^{2}}_{\alpha}(\beta)$ obtained from $\beta$ by performing $n^{2}$ Dehn-twists along $\alpha$. Namely, by construction, the curve $\beta'$ intersects orthogonally the curve $\alpha$ for the metric $h_{n}'$, thus the same estimate as in Equation \eqref{eq:stima asintotica curva} holds for the length of the curve $\beta'$ with respect to the metric $h'_{n}$. 
\end{proof}

\subsection{Discussion of the optimality}

In this subsection, we will construct some examples to show that the result of Theorem \ref{cor:confrontovolumeThurston} is optimal, in some sense. The first situation we consider is the case of a sequence of manifolds $M_{h,h'}$ for which one metric is fixed, and the other metric diverges in $\Teich(S)$. In this case, the volume of $M_{h,h'}$ is in fact bounded also from below by the exponential of Thurston's asymmetric distance:

\begin{prop} \label{prop lower bound th}
Let $\Omega$ be a compact set in $\Teich(S)$ and let $M_{h,h'}$ be a maximal globally hyperbolic $\AdS^{3}$ manifold with $h\in\Omega$. There exists a constant $C=C(\Omega)>0$ such that 
\[
	C(\Omega)\exp(\dth(h,h'))-C(\Omega)\leq \Vol(\mathcal{C}(M_{h,h'}))
\]
for every $h\in\Omega$ and every $h'\in\Teich(S)$.
\end{prop}
\begin{proof}
Let $h'$ be any hyperbolic metric on $S$, let $\lambda$ be the measured lamination such that $h'=E_l^\lambda(h)$, and let $\alpha$ be any simple closed curve on $S$. By a simple formula, we have:
$$\ell_\alpha(h')\leq \ell_\alpha(h)+\iota(\lambda,\alpha)~.$$
In fact, it is easy to check that this formula is true when $\lambda$ is a simple closed curve, since the $h'$-geodesic representative of $\alpha$ is shorter than the piecewise-geodesic curve obtained by glueing the image of the $h$-geodesic representative of $\alpha$ and subintervals of the simple closed curve $\lambda$ according to the earthquake measure.
The general case follows by a continuity argument. Hence we have:
$$\frac{\ell_\alpha(h')}{\ell_\alpha(h)}\leq 1+\frac{\iota(\lambda,\alpha)}{\ell_\alpha(h)}~.$$ 
We claim that there exists a constant $C=C(\Omega)>0$ such that, for every pair of measured laminations $\mu,\lambda\in\mathcal{ML}(S)$,
$$\frac{\iota(\mu,\lambda)}{\ell_\mu(h)\ell_\lambda(h)}\leq D(\Omega)~.$$
The proof will then follows directly from the claim, since we will then have
$$\frac{\ell_\alpha(h')}{\ell_\alpha(h)}\leq 1+D(\Omega)\ell_\lambda(h)\leq 1+4D(\Omega)\Vol(\mathcal{C}(M_{h,h'}))$$
by Theorem \ref{prop:comparison volume lams}, for every simple closed curve $\alpha$. Therefore (recall Equation \eqref{eq thuston lipschitz})),
$$C(\Omega)\exp(\dth(h,h'))-C(\Omega)=C(\Omega)\sup_\alpha \frac{\ell_\alpha(h')}{\ell_\alpha(h)}-C(\Omega)\leq \Vol(\mathcal{C}(M_{h,h'}))~,$$
where $C(\Omega)=1/4D(\Omega)$.

To prove the claim, suppose by contradiction there exists no such constant $D(\Omega)$, and therefore there exist a sequence $h_n\in\Omega$, and sequences $\mu_n,\lambda_n\in\mathcal{ML}(S)$ such that
$$\frac{\iota(\mu_n,\lambda_n)}{\ell_{\mu_n}(h_n)\ell_{\lambda_n}(h_n)}\to+\infty~.$$
Now, up to extracting subsequences, we can assume $h_n\to h_\infty\in\Omega\subset\Teich(S)$. Moreover, by the compactness of the space of projective measured laminations on $S$, we can assume that there exist $a_n,b_n>0$ such that  $a_n\mu_n\to\mu_\infty$ and $b_n\lambda_n\to\lambda_\infty$, for $\mu_\infty,\lambda_\infty\neq 0$. This leads to a contradiction, as 
$$\frac{\iota(a_n\mu_n,b_n\lambda_n)}{\ell_{a_n\mu_n}(h_n)\ell_{b_n\lambda_n}(h_n)}=\frac{\iota(\mu_n,\lambda_n)}{\ell_{\mu_n}(h_n)\ell_{\lambda_n}(h_n)}\to\frac{\iota(\mu_\infty,\lambda_\infty)}{\ell_{\mu_\infty}(h_\infty)\ell_{\lambda_\infty}(h_\infty)}<+\infty$$
since the quantities $\ell$ and $\iota$ vary with continuity.
\end{proof}

Recall that the action of the mapping class group of $S$ on 
$$\Teich_\epsilon(S):=\{ h \in \Teich(S) \ | \ \injrad(h)\geq \epsilon\}$$
is co-compact, by \cite{MR0276410}. As the volume $\Vol(\mathcal{C}(M_{h,h'})$ is invariant under the diagonal action of the mapping class group on $\Teich(S)\times\Teich(S)$, we deduce the following stronger version of Proposition \ref{prop lower bound th}.

\begin{cor}
Given any $\epsilon>0$, there exists a constant $C=C(\epsilon)$ such that
\[
	C(\epsilon)\exp(\dth(h,h'))-C(\epsilon)\leq \Vol(\mathcal{C}(M_{h,h'}))
\]
for every $h\in\Teich_\epsilon(S)$ and every $h'\in\Teich(S)$, where $\Teich_\epsilon(S)$ is the $\epsilon$-thick part of Teichm\"uller space of $S$.
\end{cor}

We will now discuss the optimality of the multiplicative constant in the upper bound of Theorem \ref{cor:confrontovolumeThurston}. More precisely, we will exhibit a sequence of examples, in a surface $S_g$ for any genus $g$, so that the volume grows actually like $|\chi(S_g)|\exp(\min\{\dth(h,h'), \dth(h',h)\})$.

\begin{prop} \label{prop optimality genus}
There exist universal constants $C,g_0>0$ and there exist sequences {of} hyperbolic metrics $h_g,h'_g$ in $\Teich(S_g)$, where $S_g$ is the closed orientable surface of genus $g$, such that:
$$\Vol(\mathcal{C}(M_{h_g,h'_g}))\geq C|\chi(S_g)|\exp(\dth(h_g,h_g'))$$
for every $g\geq g_0$.
\end{prop}
\begin{proof}
Fix a pants decomposition $P_g$ of $S_g$, which is composed of $3g-3$ disjoint simple closed curves $\alpha_1,\ldots,\alpha_{3g-3}$. Consider a hyperbolic metric $h_g$ for which all the the simple closed curves $\alpha_1,\ldots,\alpha_{3g-3}$ have the same length, say $u$ (independently of $g$).  

Let us consider the hyperbolic metric $h'=E_l^\lambda(h)$, where $\lambda$ is the multicurve $\alpha_1,\ldots,\alpha_{3g-3}$, where all the curves are endowed with the same weight $w>0$. Now, given any other simple closed curve $\alpha$, we have (as in the proof of Proposition \ref{prop lower bound th}):
$$\frac{\ell_\alpha(h')}{\ell_\alpha(h)}\leq 1+\frac{\iota(\lambda,\alpha)}{\ell_\alpha(h)}=1+w\frac{\iota(P,\alpha)}{\ell_\alpha(h)}~.$$ 
Now, observe that every time the curve $\alpha$ crosses a curve $\alpha_i$ of $P$, $\alpha$ needs to exit the pair of pants adjacent to $\alpha_i$ through some boundary component of the same pair of pants. Hence the length of $\alpha$ is at least the intersection number $\iota(P,\alpha)$ times the distance between two boundary components. Since we chose $h_g$ so that all pairs of pants in the decomposition have the same length $u$ for all boundary components, the distance between two boundary components can be computed, as in Proposition \ref{prop:controesempioThurston}, as:
$$r(u)=2\arcsinh\left(\frac{\cosh(\frac{u}{4})}{\sinh(\frac{u}{2})}\right)=2\arcsinh\left(\frac{1}{2\sinh(\frac{u}{4})}\right)~.$$
Thus we obtain
$$\ell_\alpha(h)\geq \iota(P,\alpha)\cdot r(u)~.$$ 
On the other hand, observe that $\ell_\lambda(h)=w(3g-3)u$. 
Hence we get:
$$\frac{\ell_\alpha(h')}{\ell_\alpha(h)}\leq 1+w\frac{\iota(P,\alpha)}{\iota(P,\alpha)r(u)}=1+\frac{w}{r(u)}=1+\frac{\ell_\lambda(h)}{(3g-3)ur(u)}~.$$
Since this inequality holds for every simple closed curve $\alpha$, recalling Equation \eqref{eq thuston lipschitz}, we obtain:
$$\exp(\dth(h_g,h_g'))\leq 1+\frac{2}{(3ur(u))}\frac{1}{|\chi(S_g)|}\ell_\lambda(h)\leq  1+\frac{8}{(3ur(u))}\frac{1}{|\chi(S_g)|}\Vol(\mathcal{C}(M_{h_g,h'_g}))~,$$
where in the last step we have used Theorem \ref{prop:comparison volume lams}. In particular this shows that
$$\frac{\Vol(\mathcal{C}(M_{h_g,h'_g}))}{\exp(\dth(h_g,h_g'))-1}\geq C_0|\chi(S_g)|~,$$
for some constant $C_0>0$. Since $\ell_\lambda(h)$ (and thus also the volume) is going to infinity, it follows that
$$\Vol(\mathcal{C}(M_{h_g,h'_g}))\geq C|\chi(S_g)|\exp(\dth(h_g,h_g'))$$
for every constant $C<C_0$, if $g\geq g_0$. This concludes the claim.
\end{proof}
\begin{remark}
In the proof of Proposition \ref{prop optimality genus}, the constant $C_0$ can actually be chosen as:
$$C_0=\frac{3}{8}\max_{u\in(0,\infty)}ur(u)~.$$
A numerical computation shows $C_0\sim 1.30\ldots$, and thus $C$ can be chosen arbitrarily close to this value.
\end{remark}

\begin{remark}
The proof of Proposition \ref{prop optimality genus} actually produces sequences $h_g,h_g'$ such that 
$$\Vol(\mathcal{C}(M_{h_g,h'_g}))\geq C|\chi(S_g)|\exp\left(\max\{\dth(h_g,h_g'),\dth(h_g',h_g)\}\right)~.$$
In fact, $h_g$ was chosen so that all pairs of pants in the pants decomposition $P$ have a certain shape, and $h_g'$ is obtained by earthquake along $P$. Hence for the metric $h_g'$, the pairs of pants also have this shape as well (in other words, $h_g$ and $h_g'$ only differ by twist coordinates in the Fenchel-Nielsen coordinates provided by $P$). Hence, switching left earthquakes with right earthquakes, the proof holds analogously for the other Thurston's distance.
\end{remark}

\section{Weil-Petersson distance}\label{sec:distanzaWP}
In this section we study the relation between the volume of a maximal globally hyperbolic $\AdS^{3}$ manifold and the Weil-Petersson distance between its left and right metric. 

\subsection{Weil-Petersson metric on Teichm\"uller space}
The Weil-Petersson metric is a Riemannian metric on Teichm\"uller space, which connects the hyperbolic and the complex geometry of surfaces. \\
\\
Given a Riemann surface $(S,X)$, let us denote by $K$ the canonical line bundle of $S$, that is the holomorphic cotangent bundle. It is known that the vector space $QD(X)=H^0(S,K^2)$ of holomorphic quadratic differentials on $(S,X)$ has complex dimension $3g-3$ and can be identified with the cotangent space $T^{*}_{[X]}\Teich(S)$. We recall briefly this identification for the convenience of the reader.
\\
\indent A Beltrami differential $\mu$ is a smooth section of the vector bundle $\overline{K}\otimes K^{-1}$. In local coordinates, we can write $\mu=\mu(z)\frac{d\overline{z}}{dz}$. Beltrami differentials can be interpreted as $(0,1)$-forms with value in the tangent bundle of $S$ and correspond to infinitesimal deformations of the complex structure $X$. If we denote with $BD(X)$ the vector space of Beltrami differentials and with $BD_{tr}(X)$ the subspace corresponding to trivial deformations of the complex structure $X$, we have an identification:
\[
	T_{[X]}\Teich(S) \cong BD(X)/BD_{tr}(X) \ .
\]
The duality pairing between a Beltrami differential $\mu$ and a holomorphic quadratic differential $\Phi$
\[
	\langle \mu , \Phi \rangle = \int_{S} \mu(z)\phi(z) dz\wedge d\bar{z} \ ,
\]
where in local coordinates $\Phi=\phi(z)dz^{2}$, induces the aforementioned isomorphism $QD(X)\cong T^{*}_{[X]}\Teich(S)$.

Let $h$ be the unique hyperbolic metric on $S$ compatible with the complex structure $X$. If we write in local coordinates $h=\sigma_{0}^{2}(z)|dz|^{2}$, the Weil-Petersson metric on $\Teich(S)$ arises from the real part of the Hermitian product on $QD(X)$, namely:
\[
	\langle \Phi, \Psi\rangle_{\weil}=\int_{S} \frac{\phi(z)\overline{\psi}(z)}{\sigma_{0}^{2}(z)}dz\wedge d\overline{z}
\]
via the above duality pairing. \\
\\ 
\indent The Weil-Petersson metric is geodesically convex (\cite{Wolpert-Nielsenproblem}), it has negative sectional curvature (\cite{Wolpert-Chernforms}, \cite{Tromba-WPmetric}) and the mapping class group acts by isometries (\cite{WPisometry}). However, the Weil-Petersson metric is not complete (\cite{Wolpert-WPincomplete}) and its completion gives rise to the augmented Teichm\"uller space $\overline{\Teich(S)}$, obtained by adding noded  Riemann surfaces (\cite{Masur-WPcompletion}). The Weil-Petersson distance from a point $X \in \Teich(S)$ to a noded Riemann surface $Z$ with nodes along a collection of curves $\alpha_{1}, \ldots, \alpha_{k}$ is estimated by (see \cite[Section 4]{Wolpert1} and \cite[Theorem 2.1]{MR2872556})
\begin{equation}\label{eq:WPpinching}
	\dwp(X,Z)\leq \sqrt{2\pi\ell} \ ,
\end{equation}
where
\[
	\ell=\ell_{\alpha_{1}}(h)+\ldots +\ell_{\alpha_{k}}(h)
\]
is the sum of the lengths of the curves $\alpha_{j}$ computed with respect to the unique hyperbolic metric $h$ compatible with the complex structure on $X$. 

\subsection{A negative result}
The failure of completeness of the Weil-Petersson metric at limits of pinching sequences in Teichm\"uller space implies that it is not possible to find an upper-bound for the volume of a maximal globally hyperbolic $\AdS^{3}$ manifold $M_{h_{r}, h_{l}}$ in terms of the Weil-Petersson distance $\dwp(h_{l}, h_{r})$, as the following proposition shows.

\begin{prop} It is not possible to find a continuous, increasing and unbounded function $f: \R^{+} \rightarrow \R^{+}$ such that 
\[
	\Vol(M_{h_{l}, h_{r}}) \leq f(\dwp(h_{l},h_{r})) \ .
\]
\end{prop}
\begin{proof}It is sufficient to exibit a sequence of maximal globally hyperbolic $\AdS^{3}$ manifolds $M_{n}=M_{h_{n}, h_{n}'}$ such that
\[
	\lim_{n\to +\infty}\Vol(M_{n}) =+\infty \ \ \ \ \ \ \  \text{but} \ \ \ \ \ \ \ \dwp(h_{n}, h_{n}')\leq C \ \ \ \ \forall n \in \N 
\]
for some constant $C>0$.

An example can be constructed as follows. Fix a hyperbolic metric $h \in \Teich(S)$ and a pants decomposition $P=\{\alpha_{1}, \ldots, \alpha_{3g-3}\}$ of $S$. Consider a sequence of hyperbolic metrics $h_{n}'$ obtained by letting the lengths of the curves $\alpha_{j}$ go to $0$ for every $j=1, \dots, 3g-3$. By construction, the sequence $h_{n}'$ leaves every compact subset in $\Teich(S)$ and it is converging to the noded Riemann surface $Z$ in the augmented Teichm\"uller space $\overline{\Teich(S)}$ where all the curves of the pants decomposition $P$ are pinched. Therefore, by Equation (\ref{eq:WPpinching}), 
\[
	\dwp(h,h_{n}')\leq \dwp(h,Z)+\dwp(h_{n}',Z) \leq C\ ,
\]
where $C=2\sqrt{2\pi\ell_P(h)}$. On the other hand, the volume of the $\AdS^{3}$ manifolds $M_{n}$ is diverging because
\[
	\Vol(M_{n})\geq \Vol(\Omega_{n})=\frac{\pi^{2}}{2}|\chi(S)|+\frac{1}{4}\int_{S}\tr(b_{n}) d\mathrm{A}_{h}= \frac{\pi^{2}}{2}|\chi(S)|+\frac{1}{4}F(h, h_{n}')
\]
and the functional $F(h, \cdot): \Teich(S) \rightarrow \R^{+}$ is proper (\cite[Proposition 1.2]{bms}).

Notice that we can actually make the constant $C$ arbitrarily small by choosing the metric $h$ appropriately.
\end{proof}

\subsection{A lower bound on the volume}
We can bound the volume of a maximal globally hyperbolic $\AdS^{3}$ manifold in terms of the Weil-Petersson distance between its left and right metric from below. 

\begin{reptheorem}{thm:stimaWP} Let $M_{h,h'}$ be a maximal globally hyperbolic $\AdS^{3}$ manifold. Then there exist some positive constants $a,b,c>0$  such that
\[
	\exp\left({\frac{a}{|\chi(S)|}\dwp(h,h')-b|\chi(S)|}\right)-c\leq \Vol(\mathcal C(M_{h,h'}))~.
\]
\end{reptheorem}


The proof relies on a precise estimate of the norm of the Weil-Petersson gradient of the length function, whose proof is postposed to the next section.

\begin{reptheorem}{thm:mainestimate} 
There exists a universal constant $a>0$ such that for every $\lambda \in \mathcal{ML}(S)$ and for every $h \in \Teich(S)$, we have
\begin{equation}\label{eq stima gradiente}
	\| \grad \ell_{\lambda}(h)\|_{\weil}\geq \frac{a}{|\chi(S)|}\ell_{\lambda}(h) \ .
\end{equation}
\end{reptheorem}

We will also need the following result by Bers (\cite{bers}, see also \cite[Theorem 5.13,5.14]{buser}):
\begin{theorem} \label{thm bers} Let $S$ be a closed surface of genus $g\geq 2$. For every hyperbolic metric $h \in \Teich(S)$ there is a pants decomposition $P$ such that $\ell_{\alpha}(h)<L_g$ for every $\alpha \in P$, {where $L_g=6\sqrt{3\pi}(g-1)=3\sqrt{3\pi}|\chi(S)|$. }
\end{theorem}

We will refer to the constant $L_g$ as Bers' constant. Given $h\in \Teich(S)$ we can find a pants decomposition $P=\{\alpha_{1}, \dots, \alpha_{3g-3}\}$, such that $\ell_{\alpha_{j}}(h)<L_g$ for every $j=1, \dots 3g-3$. If we perturb the metric $h$ using an earthquake, we can estimate how the lengths of the curves $\alpha_{j}$ change only in terms of the Bers' constant and of the length of the lamination.  

\begin{lemma}\label{lm:stimalunghezze} Let $h,h'\in \Teich(S)$. Let $\lambda$ be the measured geodesic lamination such that $h'=E_{l}^{\lambda}(h)$. Fix a pants decomposition $P=\{\alpha_{1}, \ldots, \alpha_{3g-3}\}$ such that $\ell_{\alpha_{j}}(h)<L$ for every $j=1, \dots 3g-3$. Then there exists a constant $d(L)>0$ depending only on $L$ such that
\[
	\ell_{\alpha_{j}}(h')\leq L+\frac{\ell_{\lambda}(h)}{d(L)} \ ,
\]
for every $j=1, \ldots, 3g-3$. 
\end{lemma}
\begin{proof}It is well known (\cite{Nielsenproblem}) that the first variation of the length of a simple closed curve $\gamma$ 
along an earthquake path is given by the integral over $\gamma$ of the cosines of the angles formed by $\gamma$ with the lamination $\lambda$. As a consequence, 
\[
	\left|\left.\frac{d}{dt}\right|_{t=t_{0}}\ell_{\gamma}(E_{l}^{t\lambda}(h))\right|\leq \iota(\gamma,\lambda) \ ,
\]
for every $t_0$. Hence
\[
	|\ell_{\gamma}(h')-\ell_{\gamma}(h)|\leq \iota(\lambda, \gamma) ~.
\]
Therefore for every $j=1, \dots 3g-3$ we can give an upper-bound for the lengths of the curves $\alpha_{j}$:
\begin{equation}\label{eq:1a}
	\ell_{\alpha_{j}}(h')\leq \ell_{\alpha_{j}}(h)+\iota(\alpha_{j}, \lambda)\leq L+\iota(\alpha_{j}, \lambda) \ .
\end{equation}
We only need to estimate the intersection between the curves $\alpha_{j}$ and the lamination $\lambda$ in terms of the length of the lamination. We claim that
$$\ell_{\lambda}(h)\geq d\cdot\iota(\lambda, \alpha_{j})$$
for some constant $d=d(L)$. To prove the claim,
suppose first that $\lambda=(c,w)$ consists of a weighted simple closed geodesic. By the Collar Lemma, since $\ell_{h}(\alpha_{j})\leq L$, there exist disjoint tubular neighborhoods $T_{\alpha_j,d(L)}$ of the geodesics $\alpha_j$ of width 
\begin{equation} \label{eq ciao}
	d(L)=\arcsinh\left(\frac{1}{\sinh\left(\frac{L}{2}\right)}\right) \ .
\end{equation}
The intersection of $c$ with $T_{\alpha_j,d(L)}$ is the disjoint union of $\#(c \cap \alpha_{j})$ geodesic arcs of length at least $d(L)$.
We deduce that for every $j=1, \ldots, 3g-3$ we have
\begin{equation}\label{eq:2a}
	\ell_{\lambda}(h)=w\ell_{c}(h)\geq wd(L)\sum_{j=1}^{3g-3}\#(c \cap \alpha_{j}) = d(L)\iota(\lambda, \alpha_{j}) \ .
\end{equation}
The general case of the claim follows by a standard approximation argument using the well-known fact that weighted simple closed curves are dense in the space of measured geodesic laminations. The proof then follows by combining Equation (\ref{eq:1a}) and Equation (\ref{eq:2a}).
\end{proof}

Given a pants decomposition $P=\{\alpha_{1}, \dots, \alpha_{3g-3}\}$ and a real number $L>0$, we define
\[
   V_{L}(P)=\{ h \in \Teich(S) \ | \ \ell_{\alpha_{j}}(h)\leq L \ \text{for every} \ j=1, \dots, 3g-3\}
\]

\begin{prop}\cite[Proposition 2.2]{BrockWPconvexcore}\label{prop:bounddiam} For every pants decomposition, the set $V_{L}(P)$ has bounded diameter for the Weil-Petersson metric. More precisely, for every pants decomposition $P$ of $S$,
\[
	\diam_{\weil}(V_{L}(P))\leq 2\sqrt{2\pi L}~.
\]
\end{prop}

We can estimate the Weil-Petersson distance between points lying in different level sets $V_{m}(P)$.
\begin{prop}\label{prop:boundWPdistance}Let $h_{0} \in V_{m}(P)$, for some $m>L$. Then
\[
	\dwp(h_{0}, V_{L}(P)) \leq \frac{|\chi(S)|}{a}\log\left(\frac{m(3g-3)}{L}\right) \ ,
\]
where $a$ is the constant provided by Theorem \ref{thm:mainestimate}.
\end{prop}
\begin{proof} Let us denote with $\ell_P:\Teich(S) \rightarrow \R^{+}$ the function
\[
	\ell_{P}(h)=\sum_{i=1}^{3g-3}\ell_{\alpha_{i}}(h)
\]
which computes the total length of the curves $\alpha_{i}$ in the pants decomposition $P$.  In the above notation, $P$ is considered as a measured lamination, composed of the multicurve $\alpha_1,\ldots,\alpha_{3g-3}$, each with unit weight. By Theorem \ref{thm:mainestimate} we have
\[
	\| \grad \ell_{P}\|_{\weil}\geq \frac{a}{|\chi(S)|}\ell_{P}, 
\]
thus
\[
	\|\grad (\log\ell_P)\|_{\weil} \geq \frac{a}{|\chi(S)|}\ .
\]
Let $X$ be the vector field on $\Teich(S)$ defined by
\[
	X=-\frac{\grad (\log\ell_P)}{\|\grad (\log\ell_P)\|_{\weil}}
\]
and let $\gamma$ be an integral curve of $X$ such that $\gamma(0)=h_{0}$. By the previous estimates, the function $\phi(t)=(\log\ell_P)(\gamma(t))$ satisfies the differential equation
\[
	\phi'(t)=\langle \grad (\log\ell_P), \gamma'(t) \rangle_{\weil}=-\|\grad (\log\ell_P)\|_{\weil} \leq -\frac{a}{|\chi(S)|}\ .
\]
We deduce that
\[
	\phi(t)\leq \phi(0)-\frac{at}{|\chi(S)|}\leq \log(m(3g-3))-\frac{at}{|\chi(S)|}~,
\]
and that the curve $\gamma(t)$ intersects the set $V_{L}(P)$ after a time
\[
	t_{0}\leq \frac{|\chi(S)|}{a}\log\left(\frac{m(3g-3)}{L}\right) \ ,
\]
which implies the claim.
\end{proof}

We have now all the ingredients to prove Theorem \ref{thm:stimaWP}:
\begin{proof}[Proof of Theorem \ref{thm:stimaWP}] 
Let $h$ be a hyperbolic metric on $S$ and $h'=E_{l}^{\lambda}(h)$. Fix a pants decomposition $P$ such that $h \in V_{L_g}(P)$, {where $L_g$ is as in the statement of Theorem \ref{thm bers}}. By Proposition \ref{prop:bounddiam} and Proposition \ref{prop:boundWPdistance} we have

\begin{align*}
	\dwp(h,h')&\leq \dwp(h', V_{L_g}(P))+\diam_{\weil}V_{L_g}(P) \\
	&\leq \frac{|\chi(S)|}{a}\log\left(\frac{m(3g-3)}{L_g}\right)+2\sqrt{2\pi L_g} \\
	& \leq \frac{|\chi(S)|}{a}\log\left(\frac{m}{2\sqrt{3\pi}}\right)+2\sqrt{2\pi L_g} ~,
\end{align*}
for some $m \in \R$ such that $h' \in V_{m}(P)$. By Lemma \ref{lm:stimalunghezze}, we can choose $m$ such that
\[
	m\leq L_g+\frac{\ell_{\lambda}(h)}{d(L_g)} ~.
\]
Hence
\[
	2\sqrt{3\pi}{d(L_g)}\exp\left({\frac{a}{|\chi(S)|}(\dwp(h,h')-2\sqrt{2\pi L_g})}\right)-d(L_g)L_g\leq \ell_{\lambda}(h)~.
\]
Now from Equation \eqref{eq ciao}, 
$$\exp(-\delta-2\sqrt{3\pi}(g-1))\leq d(L_g)\leq \exp(-2\sqrt{3\pi}(g-1))$$
for some constant $\delta$, and thus (using again the definition of $L_g$)
$$2\sqrt{3\pi}{d(L_g)}\exp\left({\frac{a}{|\chi(S)|}\dwp(h,h')-2\sqrt{2\pi L_g})}\right)\geq \exp\left({\frac{a}{|\chi(S)|}\dwp(h,h')}-b|\chi(S)|\right)~,$$
for some constant $b>0$. In conclusion, since $d(L_g)L_g\to 0$ as $g\to\infty$, there is a constant $c>0$ such that
\[
	\exp\left({\frac{a}{|\chi(S)|}\dwp(h,h')-b|\chi(S)|}\right)-c\leq \ell_{\lambda}(h)~.
\]
The main statement of Theorem \ref{thm:stimaWP} then follows by applying Theorem \ref{prop:comparison volume lams}, up to changing the constants $b$ and $c$.
\end{proof}

\section{Gradient of length function}\label{sec:stimagradiente}
This section is devoted to the proof of Theorem \ref{thm:mainestimate}, which we recall here:

\begin{reptheorem}{thm:mainestimate} There exists a universal constant $a>0$ such that for every $\lambda \in \mathcal{ML}(S)$ and for every $h \in \Teich(S)$, the following estimate holds:
\begin{equation}\label{eq stima gradiente}
	\| \grad \ell_{\lambda}(h)\|_{\weil}\geq \frac{a}{|\chi(S)|}\ell_{\lambda}(h) \ .
\end{equation}
\end{reptheorem}

\indent First, it suffices to prove the inequality \eqref{eq stima gradiente} when $\lambda$ is a simple closed curve (with weight 1). In fact, the inequality \eqref{eq stima gradiente} is homogeneous with respect to multiplication of $\lambda$ by some positive scalar. Hence if \eqref{eq stima gradiente} holds for a simple closed curve $(c,1)$, then it holds for every $(c,w)$, where $w>0$ is any weight. In this case, the inequality then holds also for every measured geodesic lamination, since weighted simple closed curves are dense in $\mathcal{ML}(S)$, and both sides of the inequality vary with continuity.\\

\indent As an initial remark, let us notice that a lower bound of the form $\| \grad \ell_{\lambda}(h)\|_{\weil}\geq C\ell_{\lambda}(h)$ clearly holds if we restrict to the thick part of Teichm\"uller space. Namely, the function
\begin{align*} 
	g: \Teich(S)\times (\mathcal{ML}(S)\setminus \{0\}) &\rightarrow \R^{+}\\
		(h,\lambda) &\mapsto \frac{\| \grad\ell_{\lambda}(h)\|^{2}_{\weil}}{\ell^{2}_{\lambda}(h)} 
\end{align*}
is invariant under the action of the Mapping Class Group and under rescaling of the measure of $\lambda$, hence if restricted to
\[
	\Teich_{\epsilon_{0}}(S)\times (\mathcal{ML}(S)\setminus \{0\})=\{ h \in \Teich(S) \ | \ \injrad(h)\geq \epsilon_{0}\}\times (\mathcal{ML}(S) \setminus \{0\})
\]
it admits a minimum, since $\Teich_{\epsilon_{0}}(S)$ projects to a compact set in the moduli space
\[
	\mathcal{M}(S)=\Teich(S)/MCG(S) \ 
\]
and the quotient $(\mathcal{ML}(S)\setminus \{0\}) /\R^{+}$ is compact. Such argument does not provide the explicit dependence of the constant $C$ in terms of the genus of $S$, which is instead included in Theorem \ref{thm:mainestimate}. In this section we will provide a proof of Theorem \ref{thm:mainestimate} in full generality. \\

The remark above, however, motivates the fact that main difficulty will arise when dealing with hyperbolic metrics with small injectivity radius.
Let us recall that it is possible to choose a (small) constant $\epsilon_0$, such that on any hyperbolic surface $(S,h)$ of genus $g$, there are at most $3g-3$ simple closed geodesics of length at most $\epsilon_0$. We will fix such $\epsilon_{0}$ later on. Notice that any $\epsilon_0\leq 2\arcsinh(1)$ works.
By the Collar Lemma, for every simple closed geodesic $\alpha$ of length $\epsilon$, the \emph{tube}
\begin{equation} \label{eq defi tubes}
T_{\alpha,d}=\{x\in (S,h)\,|\,d_h(x,\alpha)\leq d\}~,
\end{equation}
is an embedded cylinder for any $d\leq d(\epsilon)$, where 
\begin{equation} \label{eq defi d epsilon}
d(\epsilon):=\arcsinh\left(\frac{1}{\sinh(\frac{\epsilon}{2})}\right)~.\end{equation}
Moreover, if $\alpha_1,\ldots,\alpha_{3g-3}$ are pairwise disjoint, then $T_{\alpha_1,d(\epsilon_1)},\ldots,T_{\alpha_{3g-3},d(\epsilon_{3g-3})}$ are pairwise disjoint.
 Hence we obtain a thin-thick decomposition of any hyperbolic surface $(S,h)$, that is, we have
\[
	S=S^\thin_h\cup S^\thick_h~
\]
where
\begin{equation} \label{eq:thin}
	S^\thin_h=\bigcup_i T_{\alpha_i,d(\epsilon_i)}~,
\end{equation}
where the union is over all simple closed geodesics $\alpha_i$ of length $\epsilon_i\leq\epsilon_0$, and  
\begin{equation} \label{eq:thick}
	S^\thick_h=S\setminus S^\thin_h~.
\end{equation}
It then turns out that the injectivity radius at every point $x\in S^\thick_h$ is at least $\epsilon_0/2$.
 
\subsection{Riera's formula}
We are going to prove the inequality of Equation \eqref{eq stima gradiente} for a simple closed curve $c$ on $(S,h)$. If we denote by $\gamma$ the $h$-geodesic representative of $c$, we will prove the inequality first in the case 
\[
	\mathrm{length}_{h}(\gamma\cap S^\thin_h)\leq \mathrm{length}_{h}(\gamma\cap S^\thick_h)~,
\]
and then in the opposite case, provided $\epsilon_{0}$ is small enough. In both cases, a key tool will be the following theorem. This was proved by Riera in \cite{riera} in a more general setting; the statement below is specialized to the case of closed surfaces.

\begin{theorem} \label{thm riera}
Given a closed hyperbolic surface $(S,h)$, let us fix a metric universal cover $\pi:\Hyp^2\to (S,h)$, which thus identifies $\pi_1(S)$ to a Fuchsian subgroup of $\isom(\Hyp^2)$. Given a simple closed curve $c$ in $S$, let $C \in \pi_1(S)$ be an element freely homotopic to $c$. Then
\begin{equation}\label{eq:formulaWolpert}
	\| \grad\ell_{c}(h)\|_{\weil}^{2}=\frac{2}{\pi}\ell_{c}(h)+\frac{2}{\pi}\sum_{\substack{D\in\langle C\rangle\! \textbackslash\! \pi_1(S)\! /\!\langle C\rangle \\ D\neq [\mathrm{id}]} } \left(u(D)\log\left(\frac{u(D)+1}{u(D)-1}\right)-2\right) \ ,
\end{equation}
where for $D\in \langle C\rangle \textbackslash \pi_1(S) /\langle C\rangle$ (not in the double coset of the identity) the function $u$ is defined as
\begin{equation} \label{eq defi u riera}
	u(D)=\cosh (d(\mathrm{Axis}(C),\mathrm{Axis}(DCD^{-1}))~.
\end{equation}
\end{theorem}

First of all, observe that the function $u$ in Equation \eqref{eq defi u riera} is well-defined, since if $D'=ADB$ for $A,B\in\langle C\rangle$, then 
\begin{equation} \label{eq well-defined axis}
\mathrm{Axis}(D'C{D'}^{-1})=\mathrm{Axis}(ADCD^{-1}A^{-1})~.
\end{equation}
Thus
$$d(\mathrm{Axis}(C),\mathrm{Axis}(DCD^{-1})=d(\mathrm{Axis}(C),\mathrm{Axis}(D'C{D'}^{-1})~,$$
since $A$ stabilizes the axis of $C$.

Another equivalent way to express the summation in Equation \eqref{eq:formulaWolpert} is the following. Let $\gamma$ be the $h$-geodesic representative of $c$ in $S$. Let $\mathcal G(\Hyp^2)$ be the set of (unoriented) geodesics of $\Hyp^2$ and let 
\begin{equation} \label{eq defi A}
	\mathcal A=\{(\tilde \gamma_1,\tilde \gamma_2)\in\mathcal G(\Hyp^2)\times\mathcal G(\Hyp^2):\pi(\tilde \gamma_1)=\pi(\tilde \gamma_2)=\gamma\}/\pi_1(S)~,
\end{equation}
where $\pi_1(S)$ acts diagonally on pairs $(\tilde \gamma_1,\tilde \gamma_2)$.

The set $\mathcal A$ is in bijection with $\langle C\rangle \textbackslash \pi_1(S) /\langle C\rangle$, by means of the function:
$$[D]\mapsto (\mathrm{Axis}(C),\mathrm{Axis}(DCD^{-1}))~,$$
which is well-defined since, if  $D'=ADB$ for $A,B\in\langle C\rangle$, then from Equation \eqref{eq well-defined axis}, 
$$(\mathrm{Axis}(C),\mathrm{Axis}(D'CD'^{-1}))=A\cdot (\mathrm{Axis}(C),\mathrm{Axis}(DCD^{-1}))~.$$
The map is easily seen to be surjective since for every pair of geodesics $(\tilde \gamma_1,\tilde \gamma_2)$ projecting to $\gamma$, up to composing with an element in $\pi_1(S)$ one can find a representative with $\tilde \gamma_1=\mathrm{Axis}(C)$. Finally, it is injective since, supposing
$$(\mathrm{Axis}(C),\mathrm{Axis}(DCD^{-1}))=A\cdot(\mathrm{Axis}(C),\mathrm{Axis}(D'CD'^{-1}))~,$$
this implies that $A$ stabilizes $\mathrm{Axis}(C)$ (namely, $A\in\langle C\rangle)$ and that 
$$\mathrm{Axis}(DCD^{-1})=\mathrm{Axis}(AD'CD'^{-1}A^{-1})~,$$
that is, $D^{-1}AD'\in \mathrm{Stab}(\mathrm{Axis}(C))=\langle C\rangle$ and therefore $D=AD'B^{-1}$ for some $B\in\langle C\rangle$.

Let us now observe that there is a well defined function 
$$\mathfrak u:\mathcal A\to [1,+\infty)$$
 such that $\mathfrak u[\tilde \gamma_1,\tilde \gamma_2]=\cosh d(\tilde \gamma_1,\tilde \gamma_2)$. Moreover the bijection between $\langle C\rangle \textbackslash \pi_1(S) /\langle C\rangle$ and $\mathcal A$ transforms $u$ in $\mathfrak u$.
 In conclusion, the summation of  Equation \eqref{eq:formulaWolpert} is equal to:
\begin{equation} \label{eq riera 2}
	\| \grad\ell_{c}(h)\|_{\weil}^{2}=\frac{2}{\pi}\ell_{c}(h)+\frac{2}{\pi}\sum_{\Gamma\in\mathcal A\setminus \Delta} \left(\mathfrak u(\Gamma)\log\left(\frac{\mathfrak u(\Gamma)+1}{\mathfrak u(\Gamma)-1}\right)-2\right) ~.
\end{equation}
where $\Delta\in\mathcal A$ denotes the class of $(\tilde \gamma_1,\tilde \gamma_1)$.

\subsection{Estimates in the thick part of the hyperbolic surface}

Let us begin with the case in which $\mathrm{length}_{h}(\gamma\cap S^\thin_h)\leq \mathrm{length}_{h}(\gamma\cap S^\thick_h)$. In this case, the proof will use the following preliminary lemma:

\begin{lemma} \label{lemma grad arcs}
There exist $\epsilon_0>0$ small enough and $n_0>0$ large enough such that, for every choice of:
\begin{itemize}
\item A hyperbolic metric $h$ on a closed orientable surface $S$;
\item A number $\delta>0$;
\item An embedded $h$-geodesic arc $\alpha$ of length at most $\epsilon_0$, such that the $\delta$-neighborhood of $\alpha$ is embedded;
\item A simple closed curve $c$, whose $h$-geodesic representative $\gamma$ intersects $\alpha$ at least $n_0$ times;
\end{itemize}
one has:
$$||\grad\ell_c(h)||^2_\weil\geq C\delta (\#(\alpha\cap\gamma))^2~,$$
for some constant $C=C(n_0)$ (independent on the genus of $S$). 
\end{lemma}
\begin{proof}
Recall that $\gamma$ denotes the $h$-geodesic representative of $c$, let $\pi:\Hyp^2\to (S,h)$ be a fixed metric universal cover, and let us fix a lift $\tilde\alpha$ of the geodesic arc $\alpha$, so that $\pi|_{\tilde\alpha}$ is a homemorphism onto $\alpha$. We suppose that $\#(\alpha\cap\gamma)>n_0>0$, and we will determine $n_0$ later on.

Let us denote
\begin{equation} \label{eq defi A alpha}
\mathcal A_{\tilde \alpha}=\{[\tilde \gamma_1,\tilde \gamma_2]\in\mathcal A:\tilde \gamma_1\cap\tilde\alpha\neq\emptyset,\tilde \gamma_2\cap\tilde\alpha\neq\emptyset\}~.
\end{equation}
Denote moreover $E=\gamma\cap\alpha$ and define a function
$$\varphi:E\times E\to \mathcal A_{\tilde\alpha}$$
such that $\varphi(p,q)=[\tilde \gamma_{p},\tilde \gamma_{q}]$, where $\tilde \gamma_p$ is the unique geodesic of $\Hyp^2$ such that $\pi(\tilde \gamma_p)=c$ and $\pi(\tilde \gamma_p\cap\tilde \alpha)=p$.
Clearly $\varphi$ is surjective and maps the diagonal in $E\times E$ to $\mathcal A_{\tilde\alpha}\cap \Delta$. 

We claim that, for $[\tilde \gamma,\tilde \gamma']\in\mathcal A$:
\begin{equation} \label{eq claim cardinality}
\sinh d[\tilde \gamma,\tilde \gamma']\leq 2(\sinh \epsilon_0){\exp\left(-\frac{\delta}{2}\cdot(\#(\varphi^{-1}[\tilde \gamma,\tilde \gamma'])-1)\right)}~.
\end{equation}

To prove the claim, suppose the cardinality of $\varphi^{-1}[\tilde \gamma,\tilde \gamma']$ is {$n+1$}. Therefore there are $n+1$ pairs $(p_i,q_i)$ such that $[\tilde \gamma_{p_i},\tilde \gamma_{q_i}]$ are all equivalent to $(\tilde \gamma,\tilde \gamma')$ in $\mathcal A$. This means that there exists $g_i\in\pi_1(S)$ such that $g_i(\tilde \gamma_{p_i},\tilde \gamma_{q_i})=(\tilde \gamma,\tilde \gamma')$. It follows that, for every $i\neq j$, the arcs $g_i(\tilde \alpha)$ and $g_j(\tilde \alpha)$ are distinct. Indeed, if $g_i(\tilde \alpha)=g_j(\tilde \alpha)$, then $g_i\circ g_j^{-1}$ would send $\tilde\alpha$ to itself and move at least one point of $\tilde\alpha$, which is impossible since $\alpha$ is embedded. Now, the arcs $g_i(\tilde \alpha)$ intersect $\tilde \gamma$ in the $n+1$ different points $g_i(p_i)$, which are at distance at least $\delta$ from one another since the $\delta$-neighborhood of $\alpha$ is embedded. Let $r$ and $r'$ be the feet of the common perpendicular of $\tilde \gamma$ and $\tilde \gamma'$. Then, at least one of the points $g_i(p_i)$, is at distance at least $n\delta/2$ from $r$.

Denote by $p_0=g_{i_0}(p_{i_0})$ be such point, and let $q_0$ be the projection of $p_0$ to $\tilde \gamma'$. 
The quadrilateral with vertices in $p_0,r,r',q_0$ is a Lambert quadrilateral, that is, it has right angles at $r,r',q_0$.
See Figure \ref{lambert}. Hence the following formula holds:
$$\sinh d(p_0,\tilde \gamma')=\cosh d(p_0,r)\sinh d(\tilde \gamma,\tilde \gamma')~.$$
This concludes the claim, since $d(p_0,\tilde\gamma')\leq \mathrm{length}(\alpha)\leq \epsilon$ and $d(p_0,r)\geq n\delta/2$, and thus
$$\sinh d(\tilde \gamma,\tilde \gamma')\leq\frac{ \sinh \epsilon_0}{\cosh(n\delta/2)}~,$$
from which the claim follows.

\begin{figure}[htbp]
\centering
\includegraphics[height=6cm]{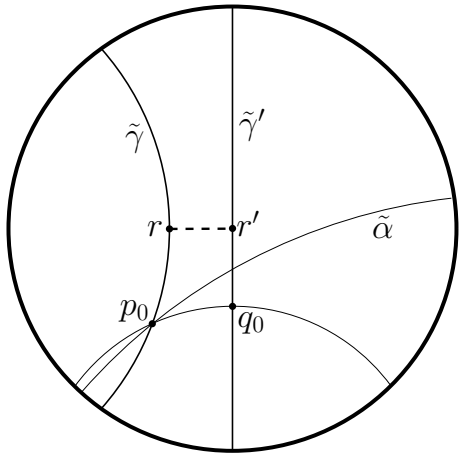}
\caption{The Lambert quadrilateral in the proof of Lemma \ref{lemma grad arcs}.} \label{lambert}
\end{figure}

Now we can conclude the proof. By Theorem \ref{thm riera}, we have:
\begin{align*} \| \grad\ell_{c}\|_{\weil}^{2}&\geq\frac{2}{\pi}\sum_{[\tilde \gamma_1,\tilde \gamma_2]\in\mathcal A_{\tilde \alpha}\setminus \Delta} \left(\cosh d[\tilde \gamma_1,\tilde \gamma_2]\log\left(\frac{\cosh d[\tilde \gamma_1,\tilde \gamma_2]+1}{\cosh d[\tilde \gamma_1,\tilde \gamma_2]-1}\right)-2\right) \\
&\geq \frac{2}{\pi}\sum_{[\tilde \gamma_1,\tilde \gamma_2]\in\mathcal A_{\tilde \alpha}\setminus \Delta} \left(\log\left(\frac{\cosh d[\tilde \gamma_1,\tilde \gamma_2]+1}{\sinh d[\tilde \gamma_1,\tilde \gamma_2]}\right)^2-2\right) \\
&\geq \frac{2}{\pi}\sum_{[\tilde \gamma_1,\tilde \gamma_2]\in\mathcal A_{\tilde \alpha}\setminus \Delta}\left(2\log 2-2\log \sinh d[\tilde \gamma_1,\tilde \gamma_2]-2\right) \\
&\geq \frac{4}{\pi}\sum_{[\tilde \gamma_1,\tilde \gamma_2]\in\mathcal A_{\tilde \alpha}\setminus \Delta}\left(-\log\sinh\epsilon_0+\frac{\delta}{2}(\#(\varphi^{-1}[\tilde \gamma,\tilde \gamma'])-1)+\log 2-1\right) \end{align*}
where in the last line we have used Equation \eqref{eq claim cardinality}. Therefore, {if we suppose that $\epsilon_{0}$ is small enough so that $C_1(\epsilon_{0}):=-\log(\sinh(\epsilon_{0}))-1+\log 2$>0}, we get 
\begin{align*} \| \grad\ell_{c}\|_{\weil}^{2}&\geq
 \frac{2\delta}{\pi}\sum_{[\tilde \gamma_1,\tilde \gamma_2]\in\mathcal A_{\tilde \alpha}\setminus \Delta}\left(\#(\varphi^{-1}[\tilde \gamma,\tilde \gamma'])\right)+C_1(\epsilon_0)(\#(\alpha\cap\gamma)-1) \\
 &\geq\frac{2\delta}{\pi}\cdot\#(E\times E\setminus \varphi^{-1}(\Delta))+C_1(\epsilon_0)(\#(\alpha\cap\gamma)-1) \\
 &=\frac{2\delta}{\pi}(\#(\alpha\cap\gamma))(\#(\alpha\cap\gamma)-1)+C_1(\epsilon_0)(\#(\alpha\cap\gamma)-1) \\
 &= \frac{2\delta}{\pi}(\#(\alpha\cap\gamma))^2+\left(C_1(\epsilon_0)-\frac{2\delta}{\pi}\right)\#(\alpha\cap\gamma)-C_1(\epsilon_0)\\
&\geq \frac{2\delta}{\pi}(\#(\alpha\cap\gamma)-1)^2 \ ,
\end{align*}
provided $\#(\alpha\cap\gamma)>1$. The last quantity is certainly larger than $C\cdot\delta\cdot (\#(\alpha\cap\gamma))^2$, if $\#(\alpha\cap\gamma)>n_0$, for some suitable choices of $n_0$ and $C=C(n_0)$.
\end{proof}

We will now replace the constant $\epsilon_0$ which gives the thin-thick decomposition (see Equations \eqref{eq:thin} and \eqref{eq:thick}) by a smaller constant (if necessary), so that $\epsilon_0$ is smaller than the constant given by Lemma \ref{lemma grad arcs}.

\begin{prop} \label{prop estimate thick}
Let $\epsilon_{0}<\arcsinh(\exp(\log(2)-1))$ a constant inducing a thin-thick decomposition of $S$. There exists a constant $a=a(\epsilon_{0})$ (independent of the genus of $S$) such that for every hyperbolic metric $h$ on $S$ and every simple closed curve $c$, if the $h$-geodesic representative $\gamma$ of $c$ satisfies:
\[
	\mathrm{length}_{h}(\gamma\cap S^\thin_h)\leq \mathrm{length}_{h}(\gamma\cap S^\thick_h)~,
\]
then
\[
	\| \grad \ell_{c}(h)\|_{\weil}\geq \frac{a}{|\chi(S)|}\ell_{c}(h) \ .
\]
\end{prop}
\begin{proof}
By an adaptation of the argument of \cite[Lemma A.1]{bonschlfixed}, there exists a constant $\beta_0$ (independent on $h$) such that the  subset 
\[
	\widehat \gamma:=\left\{x\in \gamma \left| \right. \,\#(\alpha_x^r(\epsilon_0/4)\cap \gamma)\leq \frac{\beta_0}{|\chi(S)|}\ell_c(h)\right\}\subseteq \gamma
\]
has $h$-lenght at most $\ell_c(h)/2$, where $\alpha_x^r(\epsilon_0/4)$ is the $h$-geodesic arc orthogonal to $\gamma$ starting at $x$, on the right with respect to a chosen orientation of $\gamma$, of length $\epsilon_0/4$. The constant $\beta_0$ only depends on the initial choice of $\epsilon_0$.\\
Now, in our hypothesis, since $\mathrm{length}_{h}(\gamma\cap S^\thin_h)+ \mathrm{length}_{h}(\gamma\cap S^\thick_h)=\ell_c(h)$, the length of $\gamma\cap S^\thick_h$ is at least $\ell_c(h)/2$. Therefore there exists some point $x\in(\gamma\setminus \widehat \gamma)\cap S^\thick_h$. Since $x$ is in the $\epsilon_0$-thick part of $(S,h)$, the arc $\alpha_x^r(\epsilon_0/4)$ is embedded. Moreover, the $(\epsilon_0/4)$-neighborhood of $\alpha_x^r(\epsilon_0/4)$ is embedded, for otherwise there would be a closed loop starting from $x$ of length less than $\epsilon_0$, which contradicts $x$ being in the $\epsilon_0$-thick part. 

Recall that, since by construction $x\in(\gamma\setminus \widehat \gamma)\cap S^\thick_h$, $\#(\alpha_x^r(\epsilon_0/4)\cap\gamma)\geq (\beta_0\ell_c(h))/|\chi(S)|$. Hence, from Lemma \ref{lemma grad arcs}, there exist constants $n_0>0$ and $K=K(\epsilon_0,n_0)>0$  such that
\begin{equation} \label{eq batman}
	||\grad\ell_c(h)||_\weil\geq K\#(\alpha_x^r(\epsilon_0/4)\cap\gamma)\geq \frac{K\beta_0}{|\chi(S)|} \ell_c(h)~,
\end{equation}
whenever $\#(\alpha_x^r(\epsilon_0/4)\cap\gamma)\geq n_0$, which occurs under the hypothesis that
\begin{equation} \label{eq tex willer} 
\ell_c(h)\geq \frac{n_0}{\beta_0}|\chi(S)|~.
\end{equation}
On the other hand, we have the inequality $\|\grad \ell_{c}(h)\|_{\weil}\geq \sqrt{(2/\pi)\ell_{c}(h)}$ from Equation \eqref{eq:formulaWolpert}. Observe that if 
\begin{equation} \label{eq kid} \ell_c(h)\leq \frac{n_0}{\beta_0}|\chi(S)|~,\end{equation}
then
$$\sqrt{(2/\pi)\ell_{c}(h)}\geq \sqrt{\frac{2\beta_0}{\pi n_0}}\frac{1}{|\chi(S)|^{1/2}}\ell_{c}(h)~.$$

Hence, putting together the cases \eqref{eq tex willer} and \eqref{eq kid}, there exists a constant $a>0$ such that
$$\| \grad \ell_{c}(h)\|_{\weil}\geq \frac{a}{|\chi(S)|}\ell_{c}(h)$$
thus concluding the proof.
\end{proof}

\subsection{Estimates in the thin part of the hyperbolic surface}

We are left to consider the case when 
\[
	\mathrm{length}_{h}(\gamma\cap S^\thin_h)\geq \mathrm{length}_{h}(\gamma\cap S^\thick_h)~,
\]
where $\gamma$ is the $h$-geodesic representative of $c$.
For this purpose, suppose $\gamma$ enters into a tube $T_{\alpha,d(\epsilon)}$, where $\alpha$ is a simple closed geodesic of length $\ell_h(\alpha)=\epsilon\leq \epsilon_0$.


Let us fix a metric universal cover $\pi:\Hyp^2\to(S,h)$ and a lift $\tilde\alpha$ of $\alpha$, that is, an entire geodesic in $\Hyp^2$. Let $A\in\pi_1(S)$ be a primitive element which corresponds to a hyperbolic isometry with axis $\tilde \alpha$. 
We will denote (in analogy with the notation of \eqref{eq defi A alpha}, but with the difference that here $\tilde\alpha$ covers $\alpha$):
$$\mathcal A_{\tilde \alpha}=\{[\tilde \gamma_1,\tilde \gamma_2]:\tilde \gamma_1\cap\tilde\alpha\neq\emptyset,\tilde \gamma_2\cap\tilde\alpha\neq\emptyset\}~,$$
which is a subset of the set of equivalence classes defined in Equation \eqref{eq defi A}.

\begin{lemma} \label{lemma counting pairs}
Let $[\tilde \gamma_1,\tilde \gamma_2],[\tilde \gamma_1',\tilde \gamma_2']\in \mathcal A_{\tilde\alpha}$. 
If $[\tilde \gamma_1,\tilde \gamma_2]=[\tilde \gamma_1',\tilde \gamma_2']$ and $d(\tilde \gamma_1\cap\tilde\alpha,\tilde \gamma_2\cap\tilde\alpha)\geq\mathrm{length}_h(\alpha)$, then there exists $k\in\Z$ such that $\tilde \gamma_1'=A^k(\tilde \gamma_1)$ and $\tilde \gamma_2'=A^k(\tilde \gamma_2)$.
\end{lemma}
\begin{proof}
Suppose that the equivalence classes of $(\tilde \gamma_1,\tilde \gamma_2)$ and $(\tilde \gamma_1',\tilde \gamma_2')$ coincide, and there does not exist any $k\in\Z$ such that $\tilde \gamma_1'=A^k(\tilde \gamma_1)$ and $\tilde \gamma_2'=A^k(\tilde \gamma_2)$. We will then prove that $d(\tilde \gamma_1\cap\tilde\alpha,\tilde \gamma_2\cap\tilde\alpha)<\mathrm{length}_h(\alpha)$. 

We first consider the case in which $d(\tilde \gamma_1\cap\tilde\alpha,\tilde \gamma_2\cap\tilde\alpha)=\mathrm{length}_h(\alpha)$, which occurs  if $\tilde\gamma_2=A(\tilde \gamma_1)$ (or $\tilde\gamma_2=A^{-1}(\tilde \gamma_1)$, which will be completely analogous). This means that there exists $D\in\pi_1(S)$ such that $D(\tilde\gamma_i')=\tilde\gamma_i$ for $i=1,2$, but $D$ is not in the stabilizer of $\tilde\alpha$. Hence $D(\tilde\alpha)$ is a geodesic of $\Hyp^2$, different from $\tilde\alpha$, which intersects both $\tilde\gamma_1$ and $\tilde\gamma_2$. 

We can also assume that $D$ is such that $0<d(\tilde\gamma_1\cap\tilde\alpha,\tilde\gamma_1\cap D(\tilde\alpha))<d(\tilde\gamma_2\cap\tilde\alpha,\tilde\gamma_2\cap D(\tilde\alpha))$. By this assumption, and the action by isometry of $\langle A\rangle$, it follows that $A\circ D(\tilde\alpha)$ intersects $\tilde\gamma_2$ in a point which is closer to $\tilde\alpha$ than $D(\alpha)\cap\tilde\gamma_2$. On the other hand, $A\circ D(\tilde\alpha)$ either intersects $\tilde\gamma_1$ in a point which is further from $\tilde \alpha$ than $D(\tilde\alpha)\cap \tilde\gamma_1$ (by the choice of $D$), or is disjoint from $\tilde\gamma_1$. In both cases, it follows that $A\circ D(\tilde\alpha)$ must intersect $D(\tilde\alpha)$, which gives a contradiction since $\alpha$ is a simple closed geodesic. See Figure \ref{fig:contradiction}.

In the case $d(\tilde \gamma_1\cap\tilde\alpha,\tilde \gamma_2\cap\tilde\alpha)>\mathrm{length}_h(\alpha)$, we get a contradiction \emph{a fortiori}, since every translate $D(\tilde\alpha)$ which intersects $\tilde\gamma_1$ and $\tilde\gamma_2$, must also intersect $A(\tilde\gamma_1)$ (or $A^{-1}(\tilde\gamma_1)$). This gives a contradiction as in the previous paragraph.
\end{proof}

\begin{figure}[htbp]
\centering
\includegraphics[height=6cm]{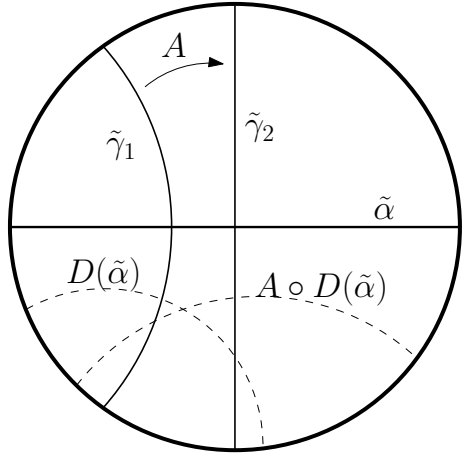}
\caption{The contradiction in the proof of Lemma \ref{lemma counting pairs}.} \label{fig:contradiction}
\end{figure}

Let us fix a connected fundamental domain $\tilde \alpha_0$ for the action of $\langle A\rangle$ on $\tilde \alpha$, and let us denote $\tilde \gamma_1,\ldots,\tilde \gamma_n$ the lifts of $\gamma$ which intersect $\tilde\alpha_0$, ordered according to an orientation of $\tilde\alpha_0$, where $n=\iota(\alpha,c)$. It follows from Lemma \ref{lemma counting pairs} and Equation \eqref{eq riera 2} that
\begin{equation} \label{eq summations}
	\| \grad\ell_{c}(h)\|_{\weil}^{2}\geq\frac{2}{\pi}\ell_{c}(h)+\frac{2}{\pi}\sum_{1\leq i\leq j\leq n} \sum_{k=1}^{+\infty} \left(\mathfrak u(\tilde\gamma_i,A^k(\tilde \gamma_j))\log\left(\frac{\mathfrak u(\tilde\gamma_i,A^k(\tilde \gamma_j))+1}{\mathfrak u(\tilde\gamma_i,A^k(\tilde \gamma_j))-1}\right)-2\right) ~.
\end{equation}

The next step thus consists of providing a uniform estimate on the multiple summation in the above inequality \eqref{eq summations}.

\begin{lemma} \label{lemma estimate from riera}
Let $\alpha$ be a simple closed geodesic on $(S,h)$ of length $\epsilon\leq \epsilon_0$ and let $\tilde \gamma_i$ and $\tilde \gamma_j$ be lifts of $\gamma$ which intersect the fundamental domain $\tilde \alpha_0$ in $\tilde \alpha$. Then there exists a universal constant $K>0$ such that
\begin{equation} \label{eq leonardo}
\sum_{k=1}^{+\infty} \left(\mathfrak u(\tilde\gamma_i,A^k(\tilde \gamma_j))\log\left(\frac{\mathfrak u(\tilde\gamma_i,A^k(\tilde \gamma_j))+1}{\mathfrak u(\tilde\gamma_i,A^k(\tilde \gamma_j))-1}\right)-2\right)\geq K\max\left\{\frac{1}{\epsilon},\frac{1}{\epsilon}|\log(\sin\theta)|^2\right\}~,
\end{equation}
where $\theta$ is the angle formed by $\tilde \gamma_i$ and $\tilde\alpha$, and
$$\mathfrak u(\tilde\gamma_i,A^k(\tilde \gamma_j))=\cosh d(\tilde\gamma_i,A^k(\tilde \gamma_j))~.$$
\end{lemma}
\begin{proof}
By a simple application of hyperbolic trigonometry, we have (see Figure \ref{fig:sinhdist}):
\begin{align*}
\sinh d(\tilde\gamma_i,A^k(\tilde \gamma_j))&\leq \sinh d(\tilde\gamma_i,A^k(\tilde \gamma_j)\cap\tilde\alpha)\\
&=(\sin\theta)
\sinh d(\tilde\gamma_i\cap\tilde\alpha,A^k(\tilde \gamma_j)\cap\tilde\alpha) \\
&\leq (\sin\theta)\sinh((k+1)\epsilon)~.
\end{align*}

\begin{figure}[htbp]
\centering
\includegraphics[height=6cm]{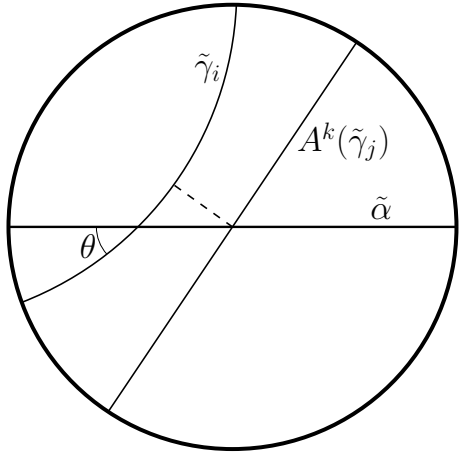}
\caption{The inequality $\sinh d(\tilde\gamma_i,A^k(\tilde \gamma_j))\leq (\sin\theta)\sinh((k+1)\epsilon)$.} \label{fig:sinhdist}
\end{figure}
Let us denote
$$F(x):=\cosh(x)\log\left(\frac{\cosh(x)+1}{\cosh(x)-1}\right)-2~,$$
which is a positive, monotone decreasing function $F:(0,+\infty)\to (0,+\infty)$.
Hence we have
$$
\sum_{k=1}^{+\infty} \left(\mathfrak u(\tilde\gamma_i,A^k(\tilde \gamma_j))\log\left(\frac{\mathfrak u(\tilde\gamma_i,A^k(\tilde \gamma_j))+1}{\mathfrak u(\tilde\gamma_i,A^k(\tilde \gamma_j))-1}\right)-2\right) 
=\sum_{k=1}^{+\infty} F(d(\tilde\gamma_i,A^k(\tilde \gamma_j)))\geq  \sum_{k=2}^{+\infty} F(\phi_\theta(k\epsilon))~,
$$
where 
$$\phi_\theta(y):=\arcsinh(\sin\theta\sinh(y))~.$$
To show that the sum in Equation \eqref{eq leonardo} is larger than $K_1/\epsilon$, we observe that $\phi_\theta(y)\leq y$ and write:
$$\sum_{k=2}^{+\infty} F(\phi_\theta(k\epsilon))\geq \int_{2}^{+\infty} F(\phi_\theta(x\epsilon))dx=\frac{1}{\epsilon}\int_{2\epsilon}^{+\infty} F(\phi_\theta(y))dy\geq \frac{1}{\epsilon}\int_{2\epsilon_0}^{+\infty} F(y)dy~.$$
This concludes the claim, by declaring
$$K_1=\int_{2\epsilon_0}^{+\infty} F(y)dy>0~.$$
In light of the inequality we have just proved, to conclude the proof it suffices to show that there exists $\theta_0>0$ such that the sum in Equation \eqref{eq leonardo} is larger than $(K_2/\epsilon)|\log(\sin\theta)|$, for all $\theta\leq \theta_0$ for some constant $K_2>0$.  

For this purpose, let us start again from 
$$\sum_{k=2}^{+\infty} F(\phi_\theta(k\epsilon))\geq \int_{2}^{+\infty} F(\phi_\theta(x\epsilon))dx\geq \frac{1}{\epsilon}\int_{2\epsilon_0}^{+\infty} F(\phi_\theta(y))dy~,$$
and observe that, by a direct analysis, there exists a constant $C>0$ such that
$$F(x)\geq C|\log(\sinh x)|$$
for $x\in(0,\arcsinh(1))$. Since $\phi_\theta(y)\in(0, \arcsinh(1))$ for $y\in(0,\arcsinh(1/\sin\theta))$, we can continue the inequality by:
\begin{align*}\sum_{k=2}^{+\infty} F(\phi_\theta(k\epsilon))&\geq  \frac{C}{\epsilon}\int_{2\epsilon_0}^{\arcsinh(\frac{1}{\sin\theta})} |\log(\sin\theta\sinh y)|dy \\
&\geq \frac{C}{\epsilon}\left(\int_{2\epsilon_0}^{\arcsinh(\frac{1}{\sin\theta})} |\log(\sin\theta)|dy-\int_{2\epsilon_0}^{\arcsinh(\frac{1}{\sin\theta})} \log(\sinh y)dy\right) \\
&\geq \frac{C}{\epsilon}\left(\int_{2\epsilon_0}^{|\log(\sin\theta)|} |\log(\sin\theta)|dy-\int_{1}^{\arcsinh(\frac{1}{\sin\theta})} ydy-C'\right) ~,
\end{align*}
where we have used that $\log(x)\leq \arcsinh(x)$, that $\log(\sinh y)\leq y$, and we put
$$C':=\int_{2\epsilon_0}^{1} |\log(\sinh y)|dy$$ Now, if we fix some small $\delta>0$, we have
$$\int_{2\epsilon_0}^{|\log(\sin\theta)|} |\log(\sin\theta)|dy=(|\log(\sin\theta)|-2\epsilon_0)|\log(\sin\theta)|\geq (1-\delta)|\log(\sin\theta)|^2$$
if $\theta$ is smaller than some $\theta_0=\theta_0(\epsilon_0)$. On the other hand, since 
\[
	\lim_{x \to +\infty}\frac{\log(x)}{\arcsinh(x)}=1~,
\]
one has $|\log(\sin\theta)|\geq (1-\delta)\arcsinh(1/\sin\theta)$, for $\theta\leq \theta_0$ (up to replacing again $\theta_0$) and therefore
$$\int_{1}^{\arcsinh(\frac{1}{\sin\theta})} ydy\leq \int_{0}^{\frac{|\log(\sin\theta)|}{1-\delta}} ydy=\frac{1}{2(1-\delta)^2}|\log(\sin\theta)|^2~.$$
In conclusion, we have 
$$\sum_{k=2}^{+\infty} F(\phi_\theta(k\epsilon))\geq\frac{C}{\epsilon}\left(\left((1-\delta)-\frac{1}{2(1-\delta)^2}\right)|\log(\sin\theta)|^2-C'\right)\geq\frac{K_2}{\epsilon}|\log(\sin\theta)|^2~,$$
for some constant $K_2$, provided $\theta\leq \theta_0$ and $\epsilon\leq\epsilon_0$. This concludes the proof.
\end{proof}

We are now ready to conclude the proof of the estimate of the Weil-Petersson gradient of the length function, in the case in which most of the length of the geodesic $\gamma$ lies in the thin part of $(S,h)$:

\begin{prop} \label{prop estimate thin}
There exists a constant $a$, depending only on the choice of a sufficiently small $\epsilon_0$ inducing a thin-thick decomposition of $S$, such that for every hyperbolic metric $h$ on $S$ and every simple closed curve $c$, if the $h$-geodesic representative $\gamma$ satisfies:
\[
	\mathrm{length}_{h}(\gamma\cap S^\thick_h)\leq \mathrm{length}_{h}(\gamma\cap S^\thin_h)~,
\]
then
\[
	\| \grad \ell_{c}(h)\|_{\weil}\geq \frac{a}{|\chi(S)|}\ell_{c}(h) \ .
\]
\end{prop}
\begin{proof}
Choosing $\epsilon_0$ small enough, we have assured that there are at most $3g-3$ simple closed geodesics $\alpha_1,\ldots,\alpha_{3g-3}$ on $(S,h)$ of length at most $\epsilon_0$. Hence, the thin part of $(S,h)$ is composed by at most $3g-3$ tubes $T_{\alpha_i,d(\epsilon_i)}$, where $\epsilon_i$ is the length of $\alpha_i$ and the tubes were defined in Equation \eqref{eq defi tubes}. Let $\alpha=\alpha_{i_0}$ be one of such simple closed geodesics, of length $\epsilon$, such that
$$\mathrm{length}_h(\gamma\cap T_{\alpha,d(\epsilon)})\geq \frac{1}{3g-3}\mathrm{length}_h(\gamma\cap S^\thin_h)\geq \frac{1}{6g-6}\ell_c(h)~.$$
We will denote $T=T_{\alpha,d(\epsilon)}$ for convenience. Observe that, for every connected component $\eta$ of $\gamma\cap T$, such that the angle formed by $\eta$ and $\alpha$ is $\theta$, we have
\begin{equation} \label{eq sinh length eta}
\sinh\left(\frac{\mathrm{length}_h(\eta)}{2}\right)= \frac{\sinh d(\epsilon)}{\sin\theta}=\frac{1}{\sin\theta\sinh(\frac{\epsilon}{2})}~,
\end{equation}
by using the definition of $d(\epsilon)$ from \eqref{eq defi d epsilon}. 

Let us choose the connected component $\eta$ whose length is minimal --- which corresponds to choosing the connected component whose angle $\theta$ of intersection with $\alpha$ is maximal. Then it is easy to see that all the other connected components have length less than $\mathrm{length}_h(\eta)+\epsilon$, since they lift to geodesic segments in $\Hyp^2$ connecting two points in the two boundary components of $\widetilde T$. See Figure \ref{fig:lengths}.

\begin{figure}[htbp]
\centering
\includegraphics[height=6cm]{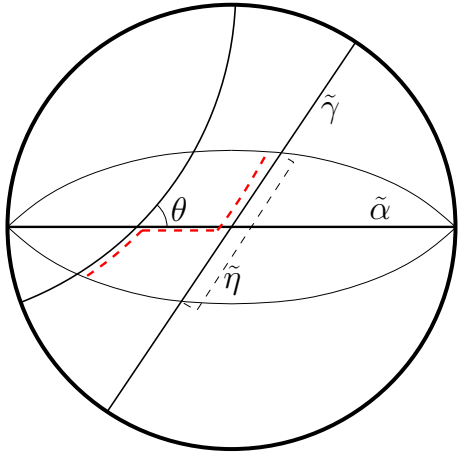}
\caption{In the universal cover, the tube $T$ is lifted to the set of points at bounded distance from $\tilde\alpha$. Using the fact that all components of $\gamma\cap T$ are disjoint, one sees that the length of every component of $\gamma\cap T$ is at most $\mathrm{length}_h(\eta)+\epsilon$, where $\eta$ is the shortest component.} \label{fig:lengths}
\end{figure}

Hence we have
\begin{equation} \label{eq ciccio}
\mathrm{length}_h(\gamma\cap T_{\alpha,d(\epsilon)})\leq \iota(\alpha,c)(\mathrm{length}_h(\eta)+\epsilon)~.
\end{equation}
On the other hand, from Equation \eqref{eq sinh length eta}, we have
$$\frac{\mathrm{length}_h(\eta)}{2}=\arcsinh\left(\frac{1}{\sin\theta\sinh(\frac{\epsilon}{2})}\right)$$
and therefore
\begin{equation} \label{eq nonna papera}
\mathrm{length}_h(\eta)+\epsilon\leq C\left|\log\left(\sin\theta\sinh\left(\frac{\epsilon}{2}\right)\right)\right|+\epsilon_0\leq C'(|\log\epsilon|+|\log(\sin\theta)|)
\end{equation}
for some suitable constants $C,C'$, if $\epsilon$ is at most some small constant $\epsilon_0$.

Now, using Equation \eqref{eq summations} and Lemma \ref{lemma estimate from riera}, we obtain
\begin{align*}\| \grad\ell_{c}(h)\|_{\weil}^{2}&\geq \frac{K}{\epsilon}\max\left\{{1},|\log(\sin\theta)|^2\right\}\iota(\alpha,c)^2 \\
&\geq \frac{K}{2\epsilon}(1+|\log(\sin\theta)|^2)\iota(\alpha,c)^2 \\
& \geq  K'\iota(\alpha,c)^2(|\log\epsilon|^2+|\log(\sin\theta)|^2) \\
&\geq \frac{K'}{2}\iota(\alpha,c)^2\left(|\log\epsilon|+|\log(\sin\theta)|\right)^2~,
\end{align*}
Therefore, comparing with \eqref{eq ciccio} and \eqref{eq nonna papera}, we have obtained 
$$\| \grad\ell_{c}(h)\|_{\weil}^{2}\geq K''(\mathrm{length}_h(\gamma\cap T_{\alpha,d(\epsilon)}))^2\geq \left(\frac{K''}{6g-6}\right)^2\ell_c(h)^2~,$$
which concludes the proof.
\end{proof}

\subsection{Conclusion of the proof and an application}

The proof of Theorem \ref{thm:mainestimate} is now straightforward:
\begin{proof}[Proof of Theorem \ref{thm:mainestimate}]
By Propositions \ref{prop estimate thick} and \ref{prop estimate thin}, we have (for a constant $a$ which replaces the constants involved there)
$$\| \grad \ell_{c}(h)\|_{\weil}\geq \frac{a}{|\chi(S)|}\max\{\mathrm{length}_{h}(\gamma\cap S^\thick_h), \mathrm{length}_{h}(\gamma\cap S^\thin_h)\} $$
and therefore 
$$\| \grad \ell_{c}(h)\|_{\weil}\geq \frac{a}{2|\chi(S)|}\ell_c(h)~,$$
as claimed.
\end{proof}

\begin{remark}
It appears possible, from the strategy used to prove Theorem \ref{thm:mainestimate}, to find the explicit value of the constant $a$ in the statement. This actually depends on some choices, for instance the constant $\epsilon_0$. We do not discuss such explicit value in this paper.
\end{remark}

We conclude by observing that, using Theorem \ref{thm:mainestimate}, one can give another proof of Theorem \ref{cor:confrontovolumeThurston}. For this purpose, first observe that, by the density of simple closed curves in the space of measured geodesic laminations, Thurston's asymmetric distance $\dth(h,h')=\inf_f\log L(f)$ can also be computed by the following characterization of $L(f)$ (compare with Equation \eqref{eq thuston lipschitz}):
\begin{equation} \label{eq thuston lipschitz 2}
	L(f)= \sup_{\mu\in\mathcal{ML}(S)}\frac{\ell_{\mu}({h'})}{\ell_{\mu}({h})}  \ .
\end{equation}
Now, fix two metrics $h$ and $h'$, and let $\lambda$ be the measured geodesic lamination such that $h'=E_l^\lambda(h)$. For any measured geodesic lamination $\mu$, by convexity of the length function along earthquake paths, we have:
\begin{equation} \label{eq convex gradient}
\ell_\mu(h')\geq \ell_\mu(h)+\left.\frac{d}{dt}\right|_{t=0}\ell_{\mu}(E_{l}^{t\lambda}(h))=\ell_\mu(h)+\langle\grad \ell_\mu,\dot E^\lambda_l(h)\rangle_\weil~,
\end{equation}
where $\dot E^\lambda_l$ defines a vector field on $\Teich(S)$.
Since it is known by a result of Wolpert (\cite{MR690844}) that the symplectic gradient of the length function $\ell_\lambda$ is the infinitesimal earthquake along $\lambda$, that is:
$$\left.\frac{d}{dt}\right|_{t=0}\ell_{\lambda}(r(t))=\omega_\weil(\dot E^\lambda_l(h),\dot r(t))=\langle J \dot E^\lambda_l(h),\dot r(t)\rangle_\weil~,$$
where $J$ is the almost-complex structure of $\Teich(S)$, from Equation \eqref{eq convex gradient} we get:
$$
\ell_\mu(h')\geq \ell_\mu(h)+\langle J\grad \ell_\mu(h),\grad \ell_\lambda(h)\rangle_\weil~.
$$
In particular, if we choose $\mu=\mu_0$, as the measured geodesic lamination such that $J\grad \ell_{\mu_0}(h)=\grad \ell_\lambda(h)$, one obtains:
\begin{align*}
L(f)\geq \frac{\ell_{\mu_0}(h')}{ \ell_{\mu_0}(h)}\geq& 1+\frac{ ||J\grad \ell_{\mu_0}(h)||_\weil ||\grad \ell_\lambda(h)||_\weil}{\ell_{\mu_0}(h)} \\
=& 1+\frac{ ||\grad \ell_{\mu_0}(h)||_\weil ||\grad \ell_\lambda(h)||_\weil}{\ell_{\mu_0}(h)}\geq 1+\frac{a^2}{|\chi(S)|^2}\ell_\lambda(h)
\end{align*}
by Theorem \ref{thm:mainestimate}, where we have also applied Equation \eqref{eq thuston lipschitz 2}. Using Theorem \ref{prop:comparison volume lams}, this concludes the alternative proof of the following:
\begin{theorem}
Let $M_{h,h'}$ be a maximal globally hyperbolic $\AdS^{3}$ manifold. Then
\[
	\Vol(\mathcal{C}(M_{h,h'})) \leq \frac{\pi^{2}}{2}|\chi(S)|+\frac{|\chi(S)|^2}{4a^2}\left(e^{\dth(h,h')}-1\right) \ .
\]
\end{theorem}

\bibliographystyle{alpha}
\bibliographystyle{ieeetr}
\bibliography{bs-bibliography}

\end{document}